\documentclass[a4paper,10pt]{article}
\usepackage{graphicx}

\usepackage[english]{babel}
\usepackage{amsfonts,amsmath,amsthm}
\usepackage{xcolor}
\newtheorem{theorem}{Theorem}[section] 
 \newtheorem{lemma}[theorem]{Lemma}

\theoremstyle{definition} 
\newtheorem{definition}[theorem]{Definition} 
 \newtheorem{remark}[theorem]{Remark}
\newtheorem{conjecture}[theorem]{Conjecture}

\newcommand{\bff}{\textup{\textbf{f}}}

\newcommand{\bfu}{\textup{\textbf{u}}}

\newcommand{\bfx}{\textup{\textbf{x}}}

\newcommand{\bftheta}{{\boldsymbol{\theta}}}

\newcommand{\e}{\textup{e}}
\newcommand{\ii}{\textup{i}} 

\newcommand{\RR}{\mathbb R}
\newcommand{\CC}{\mathbb C}
\newcommand{\ZZ}{\mathbb Z}
\newcommand{\NN}{\mathbb N} 

\newcommand{\WW}{{\cal W}}
\newcommand{\spP}{{\mathbb{P}}}
\newcommand{\spQ}{{\mathbb{Q}}}
\newcommand{\spH}{{\mathbb{H}}}
\newcommand{\spT}{{\mathbb{T}}}
\def\param{\alpha}
\def\fracparam{{\alpha/n}}
\def\paramgen{\mu}
\def\fracparamgen{{\mu/n}}

\DeclareMathOperator{\rank}{rank}

\allowdisplaybreaks[1]

\newcommand{\suppr}{\hfuzz=100pt \vfuzz=100pt} \suppr

\usepackage[bookmarks,colorlinks,citecolor=red,linkcolor=blue]{hyperref} 

\usepackage{contour}
\newcommand{\ve}{{\boldsymbol{e}}}
\newcommand{\iii}{{\boldsymbol i}} 
\newcommand{\jj}{{\boldsymbol j}}

\newcommand{\kk}{{\boldsymbol k}}

\newcommand{\mm}{{\boldsymbol m}}
\newcommand{\nn}{{\boldsymbol n}}
\newcommand{\pp}{{\boldsymbol p}}

\newcommand{\U}{{\boldsymbol U}}
\newcommand{\V}{{\boldsymbol V}}
\newcommand{\bnu}{{\boldsymbol\nu}}

\newcommand{\bbeta}{{\boldsymbol\beta}}

\newcommand{\bmu}{{\boldsymbol\mu}}
\newcommand{\bparam}{{\boldsymbol\param}} 
\newcommand{\bphi}{{\boldsymbol\phi}} 
\newcommand{\btau}{{\boldsymbol\tau}} 

\newcommand{\spQQ}{\contour[1]{black}{$\spQ$}}

\begin{document}

\title{Spectral analysis of matrices in collocation methods based on generalized B-splines with high smoothness}

%
\author{Fabio Roman}


\date{}

\maketitle


\section{Introduction} 

Although Galerkin discretizations have been intensively employed in the IgA context, an efficient implementation requires special numerical quadrature rules when constructing the system of equations; see e.g. \cite{collocbsp18}. To avoid this issue, isogeometric collocation methods have been recently introduced in \cite{collocbsp02,collocbsp22}, giving origin to a topic of study which is proceeding almost parallel to the Galerkin evolution.

In this paper we analyze the spectral properties of the matrices arising from isogeometric collocation methods based on GB-splines to approximate an elliptic PDE with variable coefficients and a generic geometry map, in a fashion which is similar to the one presented for isogeometric Galerkin methods in \cite{galerkin_gbsp}, but not limiting only to the case of constant coefficients (with the Laplacian operator for the principal terms) and a trivial geometry map.

This paper generalizes some of the results of \cite{our-MATHCOMP}, presented there for the polynomial B-splines; in particular, we remark that, on one hand, the conditioning and the extremes of the spectrum \cite{collocbsp19,colloc1,collocbsp29} cover an important role in a feasibility study, for example while evaluating the intrinsic error; but on the other hand, it is of prominent interest the possible existence of an eigenvalue distribution, in the sense of Weyl \cite{collocbsp08}. This information is used both in the convergence analysis of preconditioned Krylov methods \cite{BK,Ksirev}, and in the design and in the theoretical analysis of effective preconditioners \cite{tau-toep,colloc1} and fast multigrid or multi-iterative solvers \cite{ADS,DGMSS14a}. 

In light of this, we prove that an eigenvalue distribution actually exists, and it is compactly described by a symbol $f$, which has a canonical structure incorporating the approximation technique, the geometry $ {\bf G} $, and the coefficients of the principal terms of the PDE $K$. This canonical structure is intrinsic by any local method of approximation of PDEs: also methods like Finite Differences and Finite Elements have the same picture, and this is true also for other techniques, which give substantially the same formal structure for $f$ \cite{collocbsp04,Serra03,Serra06}. The difference is given by the information relative to the approximation technique, which is identified by a finite set of functions in the Fourier variables $ {\bftheta} := (\theta_1, \ldots, \theta_d) \in [-\pi,\pi]^d $, and that depends on the specific technique chosen. $ {\bf G} $, which is a map in the variables $ {\bf \hat{x}} := (\hat{x}_1, \ldots, \hat{x}_d) $ defined on the reference domain $ \hat{\Omega} := [0,1]^d $, and $K$, whose elements are in the physical variables $ {\bf x} = (x_1, \ldots, x_d) $ defined on the physical domain $\Omega$, are independent of the method. Note that $ {\bf G} $ and $K$ are implicitly present also in the Galerkin study of \cite{galerkin_gbsp}, where we operated in the case $ {\bf G} $ is an identity map, and $K$ is the identity matrix. 

In detail, some of the analytic features of the symbol $f$ are interesting, for example the fact that, if we consider for simplicity the unidimensional setting, there is a monotonic and exponential convergence to zero of $f$ with respect to the degree $p$ at the points $ \theta = \pm \pi $, which is inherited from \cite{our-MATHCOMP} (and holds also in the Galerkin case), but it is not in general true for a generic technique. This appearance of small eigenvalues related to high frequency eigenvectors causes a slowdown in the convergence while using classical multigrid and preconditioning techniques; a solution for the polynomial case was presented in \cite{DGMSS14c}.

This paper is organized as follows. Section \ref{sec:prel_spec} presents some preliminaries on spectral analysis, which will be pivotal in order to devise our study. Section \ref{sec:GB} is devoted to the definition and main properties of GB-splines, while Section \ref{sec:cardGB} deals specifically with the cardinal ones. Section \ref{sec:symbols} considers some functions based on cardinal GB-splines, which will play the role of symbols in our context, and 
Section \ref{sec:col1D} is devoted to the formal expression of the IgA collocation matrices based on hyperbolic and trigonometric splines, with their spectral distributions in the unidimensional case, considering both nested and non-nested spaces. Finally, Section \ref{sec:colmD} treats the study of the multidimensional case, and in Section \ref{sec:colconcl} some conclusions are formulated.

\section{Preliminaries on spectral analysis} \label{sec:prel_spec}

We introduce the fundamental definitions and theorems for developing our spectral analysis, see \cite{Golinskii-Serra}. By denoting with $ \mu_d $ the Lebesgue measure in $ \mathbb{R}^d $, we can define the \emph{spectral distribution of a sequence of matrices}.

\begin{definition}
Let $ \{ X_n \} $ be a sequence of matrices with strict increasing dimension, and let $ f : D \rightarrow \mathbb{C} $ be a measurable function, with $ D \subset \mathbb{R}^d $ measurable and such that $ 0 < \mu_d(D) < \infty $. We say that $ \{ X_n \} $ is distributed like $f$ in the sense of the eigenvalues, and we write $ \{ X_n \} \sim_{\lambda} f $, if, by considering $ C_c(\mathbb{C},\mathbb{C}) $ as the space of continuous functions $ F : \mathbb{C} \rightarrow \mathbb{C} $ with compact support, it holds:

$$ \lim_{n \rightarrow \infty} \frac{1}{d_n} \sum_{j=1}^{d_n} F(\lambda_j(X_n)) 
= \frac{1}{\mu_d(D)} \int_D F(f(x_1,\ldots,x_d)) {\rm d}x_1 \cdots {\rm d}x_d, 
\qquad \forall F \in C_c(\mathbb{C},\mathbb{C}) $$

\end{definition}

The next definition considers the concept of \emph{clustering of a sequence of matrices at a subset of $\mathbb{C}$}.

\begin{definition}
Let $ \{ X_n \} $ be a sequence of matrices with strict increasing dimension, and let $ S \subseteq \mathbb{C} $ be a non-empty closed subset of $\mathbb{C}$. We say that $ \{ X_n \} $ is strongly clustered at $S$ if the following condition is satisfied:

$$ \forall \varepsilon > 0, \quad \exists C_{\varepsilon} \mbox{ and } \exists n_{\varepsilon} : \quad
   \forall n \geq n_{\varepsilon}, \quad q_n(\varepsilon) \leq C_{\varepsilon} $$
where $ q_n(\varepsilon) $ is the number of eigenvalues of $X_n$ lying outside the $\varepsilon$-expansion $ S_{\varepsilon} $ of $S$, i.e.

$$ S_{\varepsilon} := \cup_{s \in S} 
   [\Re s - \varepsilon, \Re s + \varepsilon] \times [\Im s - \varepsilon, \Im s + \varepsilon] $$
   
\end{definition}

The next theorems give sufficient conditions for a sequence of matrices to have a given spectral distribution, and/or to be strongly clustered.

\begin{theorem} \label{th:poly01}

Let $ \{ X_n \} $ and $ \{ Y_n \} $ be two sequences of matrices with $ X_n, Y_n \in \mathbb{C}^{d_n \times d_n} $, and $ d_n < d_{n+1} $ for all $n$, such that

\begin{itemize}
\item $X_n$ is Hermitian for all $n$ and $ \{ X_n \} \sim_{\lambda} f $, where $ f : D \subset \mathbb{R}^d \rightarrow \mathbb{R} $ is a measurable function defined on the measurable set $D$ with $ 0 < \mu_d(D) < \infty $;
\item there exists a constant $C$ so that $ \Vert X_n \Vert, \Vert Y_n \Vert \leq C $ for all $n$;
\item $ \Vert Y_n \Vert_1 = o(d_n) $ as $ n \rightarrow \infty $, i.e., $ \lim_{n \rightarrow \infty} \frac{\Vert Y_n \Vert_1}{d_n} = 0 $
\end{itemize}
Set $ Z_n := X_n + Y_n $. Then $ \{ Z_n \} \sim_{\lambda} f $.

\end{theorem}

\begin{theorem} \label{th:poly02}

Let $ \{ X_n \} $ and $ \{ Y_n \} $ be two sequences of matrices with $ X_n, Y_n \in \mathbb{C}^{d_n \times d_n} $, and $ d_n < d_{n+1} $ for all $n$, such that

\begin{itemize}
\item $X_n$ is Hermitian for all $n$ and $ \{ X_n \} \sim_{\lambda} f $, where $ f : D \subset \mathbb{R}^d \rightarrow \mathbb{R} $ is a measurable function defined on the measurable set $D$ with $ 0 < \mu_d(D) < \infty $;
\item there exists a constant $C$ so that $ \Vert X_n \Vert, \Vert Y_n \Vert_1 \leq C $ for all $n$;
\end{itemize}
Set $ Z_n := X_n + Y_n $. Then $ \{ Z_n \} \sim_{\lambda} f $, and $ \{ Z_n \} $ is strongly clustered at the essential range of $f$, which coincides with the range of $f$ whenever $f$ is continuous.

\end{theorem}


\noindent With the following result, we can relate tensor-products and Toeplitz matrices. Given two functions $ f, h : [-\pi,\pi] \rightarrow \mathbb{R} $ in $ L_1([-\pi,\pi]) $, we can construct the tensor-product function, belonging to $ L_1([-\pi,\pi]^2) $:

$$  f \otimes h : [-\pi,\pi]^2 \rightarrow \mathbb{R}, \quad 
   (f \otimes h)(\theta_1,\theta_2) := f(\theta_1) h(\theta_2) $$
So, we can consider the three families of Hermitian Toeplitz matrices $ \{ T_{m_1}(f) \}, \{ T_{m_2}(h) \} $ and $ \{ T_{m_1,m_2}(f \otimes h) \} $. By computing directly we obtain a commutative property between the operation of tensor-product and the Toeplitz operator.

\begin{lemma} \label{lem:poly02}

Let $ f, h \in L_1([-\pi,\pi]) $ be real-valued functions. Then, for all $ m_1, m_2 \geq 1 $,

$$ T_{m_1}(f) \otimes T_{m_2}(h) = T_{m_1,m_2}(f \otimes h) $$

\end{lemma}

\section{GB-splines} \label{sec:GB}

\label{subsec:GB}
For $p,n\ge1$, we consider the uniform knot set
\begin{equation}\label{knots}
\{t_1,\ldots,t_{n+2p+1}\}:=\biggl\{\underbrace{0,\ldots,0}_{p+1},\frac{1}{n},\frac{2}{n},\ldots,\frac{n-1}{n},\underbrace{1,\ldots,1}_{p+1}\biggr\},
\end{equation}
and the section space
\begin{equation}\label{Puv}
\spP_{p}^{U,V}:=
\langle 1,x, \ldots, x^{p-2}, U(x), V(x)\rangle,
\quad x\in[0,1].
\end{equation}
The functions
$U, V\in C^{p-1}[0,1]$
are such that $\{U^{(p-1)},\ V^{(p-1)}\}$ is a
Chebyshev system in $[t_{i}, t_{i+1}]$,
i.e.,  any non-trivial element in the space $\langle U^{(p-1)}, V^{(p-1)}\rangle$ has at most one zero in $[t_{i}, t_{i+1}]$, $i=p+1,\ldots, p+n$.
We denote by $\widetilde U_i, \widetilde V_i$ the unique elements in
$\langle U^{(p-1)}, V^{(p-1)}\rangle$
satisfying
$$
\widetilde U_i(t_{i})=1,\ \widetilde U_i(t_{i+1})=0, \quad
\widetilde V_i(t_{i})=0,\ \widetilde V_i(t_{i+1})=1, \quad i=p+1,\ldots, p+n.
$$
Popular examples of such a space (\ref{Puv}) are
\begin{align}
\spP_{p}&:=\langle1,x, \ldots, x^{p-2}, x^{p-1}, x^p \rangle,
\label{Poly}\\
\spH_{p,\param}&:=\langle1,x, \ldots, x^{p-2}, \cosh(\param x),
\sinh(\param x)\rangle,  \quad  0<\param\in \RR,
\label{Exp}\\
\spT_{p,\param}&:=\langle1,x,
\ldots, x^{p-2}, \cos(\param x), \sin(\param x)\rangle,\quad
0<\param(t_{i+1}-t_i)<\pi.
\label{Trig}
\end{align}

The generalized spline space of degree $p$ over the knot set (\ref{knots}) and section spaces as in (\ref{Puv}) is defined by
\begin{equation}
\label{eq:GB_spline_space}
\left\{s\in C^{p-1}([0,1]):\ s|_{\left[\frac in,\frac{i+1}n\right)}\in\spP_{p}^{U,V},\ i=0,\ldots,n-1\right\}.
\end{equation}
Generalized spline spaces consisting of the particular section spaces (\ref{Poly}), (\ref{Exp}) and (\ref{Trig}) are referred to as polynomial, hyperbolic (or exponential) and trigonometric spline spaces (with phase $\param$), respectively.

Hyperbolic and trigonometric splines allow for an exact representation of conic sections as well as some transcendental curves (helix, cycloid, \ldots).
They are attractive from a geometrical point of view. Indeed, they are able to provide parameterizations of conic sections with respect to the arc length so that equally spaced points in the parameter domain correspond to equally spaced points on the described curve.

For fixed values of the involved parameters, the spaces
(\ref{Exp}) and (\ref{Trig}) have the same approximation power as the polynomial space
$\spP_p$, see \cite[Section~3]{CLM}.

\begin{definition} \label{GB-splines}
The GB-splines of degree $p$ over the knot set \eqref{knots} with sections in (\ref{Puv}) are denoted by
$N_{i,p}^{U,V}:[0,1]\rightarrow\RR$, $i=1,\ldots,n+p$,
and are defined recursively as follows\footnote{The functions $N_{i,1}^{U,V}$ may also depend on $p$ because of the definition of $\widetilde U_i, \widetilde V_i$, but we omit the parameter $p$ in order to avoid a too heavy notation. Moreover, in the most interesting cases (polynomial, hyperbolic and trigonometric GB-splines) the functions $ N_{i, 1}^{U,V}$ are independent of $p$.}.
For $p=1$,
$$
N_{i, 1}^{U,V}(x):=\begin{cases}
{\widetilde V_i(x)}, 
& {\rm if\ } x\in [t_i, t_{i+1}),
\\
{\widetilde U_{i+1}(x)}, 
& {\rm if\ } x\in [t_{i+1}, t_{i+2}),
\\
0,& {\rm elsewhere},
\end{cases}
$$
and for $p\geq2$,
$$
 N_{i,p}^{U,V}(x):=
\delta_{i,p-1}^{U,V}{\int_{0}^x N_{i,p-1}^{U,V}(s)\,{\rm d}s}-
\delta_{i+1,p-1}^{U,V}{\int_{0}^x N_{i+1,p-1}^{U,V}(s)\,{\rm d}s},
$$
where
$$\delta_{i,p}^{U,V}:= \left(\int_{0}^{1}N_{i, p}^{U,V}(s)\,{\rm d}s\right)^{-1}.$$
Fractions with zero denominators are considered to be zero, and
$
N_{i, p}^{U,V}(1):=\lim_{x\rightarrow 1^-} N_{i, p}^{U,V}(x).
$
\end{definition}
%
It is well known (see e.g. \cite{MPS-CMAME}) that the GB-splines $N_{1,p}^{U,V},\ldots, N_{n+p,p}^{U,V}$ form a basis for the generalized spline space (\ref{eq:GB_spline_space}).
The functions $N_{i,p}^{U,V}$ are non-negative and locally supported, namely
\begin{equation*} 
 {\rm supp}\bigl(N_{i,p}^{U,V}\bigr)=[t_i,t_{i+p+1}],\quad i=1,\ldots,n+p.
\end{equation*}
Moreover,
$$
N_{i,p}^{U,V}(0)=N_{i,p}^{U,V}(1)=0,\quad i=2,\ldots,n+p-1.
$$
For $p\geq2$ the GB-splines $N_{i,p}^{U,V}$, $i=1,\ldots,n+p$, form a partition of unity on $[0,1]$.

\begin{remark}\label{rmk:trig-hyp-GB}
Hyperbolic and trigonometric GB-splines (with sections in the spaces (\ref{Exp}) and (\ref{Trig}), respectively) approach the classical (polynomial) B-splines of the same degree and over the same knot set when the phase parameter $\param$ approaches $0$.
\end{remark}

\begin{remark}
GB-splines can be defined in a more general setting than Definition~\ref{GB-splines}.
In particular, completely general knot sets can be considered and the section space can be chosen differently on each knot interval $[t_i,t_{i+1}]$, see e.g. \cite{MPS-CMAME}.
\end{remark}

\section{Cardinal GB-splines} \label{sec:cardGB}

We now consider the space
\begin{equation}
\label{Puv-cardinal} %
\langle 1,t, \ldots, t^{p-2}, U(t), V(t)\rangle,\quad t\in [0,1],
\end{equation}
where
$U, V\in C^{p-1}[0,p+1]$
are such that $\{U^{(p-1)},\ V^{(p-1)}\}$ is a
Chebyshev system in $[0,1]$.
We denote by $\widetilde U, \widetilde V$ the unique elements in the space
$\langle U^{(p-1)}, V^{(p-1)}\rangle$
satisfying
\begin{equation}\label{eq:UVtilde-cardinal}
\widetilde U(0)=1,\ \widetilde U(1)=0, \quad
\widetilde V(0)=0,\ \widetilde V(1)=1.
\end{equation}
\begin{definition}
\label{GB-splines-car}
The (normalized) cardinal GB-spline of degree $p\geq 1$ over the uniform knot set $ \{ 0, 1, \ldots, p+1 \} $ with sections in (\ref{Puv-cardinal}) is denoted by $ \phi_{p}^{U,V} $ and is defined recursively as follows. For $p=1$,
\begin{equation}
\label{eq:cardinal_GB_1}
 \phi_{1}^{U,V}(t) := \delta_{1}^{U,V} \begin{cases}
  \widetilde V(t), & \rm{if\ } t \in [0,1), \\
  \widetilde U(t-1), & \rm{if\ } t \in [1,2), \\
0, & \rm{ elsewhere,}
\end{cases}
\end{equation}
where $\delta_{1}^{U,V}$ is a normalization factor given by
$$ \delta_{1}^{U,V} := \left(\int_0^1 \widetilde V(s)\, {\rm d}s + \int_1^2 \widetilde U(s-1)\, {\rm d}s\right)^{-1}.$$
For $p\geq 2$,
\begin{equation}
\label{eq:cardinal_GB_p}
\phi_{p}^{U,V}(t) := \int_0^t (\phi_{p-1}^{U,V}(s) - \phi_{p-1}^{U,V}(s-1))\, {\rm d}s.
\end{equation}
\end{definition}

The cardinal GB-splines of degree $p$ are the set of integer translates of $ \phi_{p}^{U,V}$, namely $\{ \phi_{p}^{U,V}(\cdot - k),\ k \in \mathbb{Z} \}$.

If the space in (\ref{Puv-cardinal}) is the space of algebraic polynomials  of degree less than or equal to $p$, i.e. $U(t)=t^{p-1}$ and $V(t)=t^p$, then the function defined in Definition~\ref{GB-splines-car} is the classical (polynomial) cardinal B-spline  of degree $p$, denoted  by $\phi_{p}$. The properties of the cardinal GB-spline listed in the next subsection generalize those of $\phi_{p}$, see e.g. \cite[Section~3.1]{GMPSS14}.

\subsection{Properties of cardinal GB-splines} \label{sec:card_gbsplines_prop} 
The cardinal GB-spline $ \phi_{p}^{U,V} $ is globally of class $ C^{p-1}$, and possesses some fundamental properties. Many of them are well established in the literature (see e.g. \cite{MPS-CMAME}), while the less known ones were proved in \cite{galerkin_gbsp}. Here we limit ourselves to state them, omitting proofs:

\begin{itemize}

\item \emph{Positivity:}
$$ \phi_{p}^{U,V}(t) > 0, \quad t \in (0,p+1). $$

\item \emph{Minimal support:}
\begin{equation} \label{eq:support-car}
 \phi_{p}^{U,V}(t) = 0, \quad t \not \in (0,p+1).
 \end{equation}

\item \emph{Partition of unity:}
\begin{equation} \label{eq:partunity}
\sum_{k \in \mathbb{Z}} \phi_{p}^{U,V}(t-k)
= \sum_{k=1}^p \phi_{p}^{U,V}(k) = 1, \quad p \geq 2.
\end{equation}

\item \emph{Recurrence relation for derivatives:}
\begin{equation}\label{eq:der-recurrence}
\bigl(\phi_{p}^{U,V}\bigr)^{(r)}(t) =
\bigl(\phi_{p-1}^{U,V}\bigr)^{(r-1)}(t) - \bigl(\phi_{p-1}^{U,V}\bigr)^{(r-1)}(t-1),
\quad 1\leq r\leq p-1.
\end{equation}


\item \emph{Conditional symmetry with respect to $\frac{p+1}{2}$}:
\begin{equation} \label{eq:sym_sym}
\phi_{p}^{U,V} \left(\frac{p+1}{2} + t \right) = \phi_{p}^{U,V}
\left(\frac{p+1}{2} - t \right) \quad \mbox{ if } \quad \phi_{1}^{U,V}(1+t) = \phi_{1}^{U,V}(1-t).
\end{equation}


\item \emph{Convolution relation:}
\begin{equation} \label{eq:convolution}
\phi_{p}^{U,V}(t)
 = \bigl(\phi_{p-1}^{U,V} * \phi_{0}\bigr)(t)
:= \int_{\RR} \phi_{p-1}^{U,V}(t-s) \phi_{0}(s) \,{\rm d}s
 = \int_0^1 \phi_{p-1}^{U,V}(t-s) \,{\rm d}s,
\quad p\geq2,
 \end{equation}
where $\phi_{0}(t) := \chi_{[0,1)}(t)$. Moreover,
\begin{equation} \label{eq:convolution-full}
 \phi_{p}^{U,V}(t) = \Bigl(\phi_{1}^{U,V} *\, \underbrace{\phi_{0} * \cdots * \phi_{0}}_{{p-1}}\Bigr)(t),
 \quad
 \phi_{p}(t) = \Bigl(\underbrace{\phi_{0} * \cdots * \phi_{0}}_{{p+1}}\Bigr)(t),
 \quad p\geq1.
\end{equation}


\item \emph{Inner products:}
  \begin{equation} \label{eq:inner-product}
  \int_{\RR} {\phi}_{p_1}^{U_1,V_1}(t)\, {\phi}_{p_2}^{U_2,V_2}(t+k)\,{\rm d}t
  = \Bigl(\phi_{1}^{U_1,V_1} * \phi_{1}^{U_2,V_2} * \underbrace{\phi_{0} * \cdots * \phi_{0}}_{p_1+p_2-2}\Bigr)(p_2+1-k),
  \end{equation}
  if $\phi_{1}^{U_2,V_2}(1+t) = \phi_{1}^{U_2,V_2}(1-t)$. In particular,
  \begin{align*}
  \int_{\RR} {\phi}_{p_1}^{U,V}(t)\, {\phi}_{p_2}(t+k)\,{\rm d}t
  &= \Bigl(\phi_{1}^{U,V} * \underbrace{\phi_{0} * \cdots * \phi_{0}}_{p_1+p_2}\Bigr)(p_2+1-k)
  = \phi_{p_1+p_2+1}^{U,V}(p_2+1-k),\\
  \int_{\RR} {\phi}_{p_1}(t)\, {\phi}_{p_2}(t+k)\,{\rm d}t
  &= \Bigl(\underbrace{\phi_{0} * \cdots * \phi_{0}}_{p_1+p_2+2}\Bigr)(p_2+1-k)
  = \phi_{p_1+p_2+1}(p_2+1-k).
  \end{align*}


\item \emph{Unity of integral:}
\begin{equation} \label{eq:intunity}
\int_0^{p+1} \phi_{p}^{U,V}(t) \,{\rm d}t = 1.
\end{equation}


\end{itemize}

In the following, we will focus in particular on hyperbolic and trigonometric cardinal GB-splines
with real phase parameters $\alpha$.
The hyperbolic cardinal GB-spline is denoted by
$\phi_{p}^{\spH_\param}$ and is defined by taking
$U(t):=\cosh(\param t)$ and $V(t):=\sinh(\param t)$. In this case, we have
\begin{equation}  \label{eq:uv_hyp}
\widetilde U(t) = \frac{\sinh(\param(1-t))}{\sinh(\param)}, \quad
\widetilde V(t) = \frac{\sinh(\param t)}{\sinh(\param)},
\end{equation}
satisfying (\ref{eq:UVtilde-cardinal}).
The trigonometric cardinal GB-spline is denoted by $\phi_{p}^{\spT_\param}$ and is defined by taking $U(t):=\cos(\param t)$ and $V(t):=\sin(\param t)$. In this case, we have
\begin{equation}  \label{eq:uv_trig}
\widetilde U(t) = \frac{\sin(\param(1-t))}{\sin(\param)}, \quad
\widetilde V(t) = \frac{\sin(\param t)}{\sin(\param)},
\end{equation}
satisfying (\ref{eq:UVtilde-cardinal}).
To simplify the notation, we will also use the notation $\phi_{p}^{\spQ_\param}$ if a statement holds for both $\phi_{p}^{\spH_\param}$ and $\phi_{p}^{\spT_\param}$,
but not necessarily for an arbitrary $\phi_{p}^{U,V}$.

\begin{remark} \label{rmk:trig-hyp-car-GB}
Referring to Remark~\ref{rmk:trig-hyp-GB}, hyperbolic and trigonometric cardinal GB-splines approach the (polynomial) cardinal B-spline of the same degree as the phase parameter $\param$ approaches $0$, i.e.,
$$
\lim_{\param\rightarrow 0}\phi_{p}^{\spQ_\param}(t)=\phi_{p}(t), \quad \spQ=\spH,\spT.
$$
\end{remark}

\begin{remark}
The symmetry requirement
$\phi_{1}^{U,V}(1+t) = \phi_{1}^{U,V}(1-t)$ in (\ref{eq:sym_sym}) is equivalent to
$$
\widetilde V(t)=\widetilde U(1-t), \quad t\in[0,1].
$$
This requirement is satisfied for hyperbolic and trigonometric cardinal GB-splines,
see (\ref{eq:uv_hyp})--(\ref{eq:uv_trig}), as well as for (polynomial) cardinal B-splines.
\end{remark}

\begin{remark} \label{rmk:def-agreement}
Taking into account the local support (\ref{eq:support-car}) and the unity of integral (\ref{eq:intunity}), we see that Definition~\ref{GB-splines-car} is in agreement with Definition~\ref{GB-splines} for $p\geq 2$.
In particular, let
$$
 \bigl\{ N^{\spQ_\param}_{i,p} : i = 1, \ldots, n+p \bigr\}, \quad \spQ=\spH,\spT,
$$
be a set of either hyperbolic or trigonometric GB-splines of degree $p$ defined over the knot set (\ref{knots}) with sections in either (\ref{Exp}) or (\ref{Trig}), respectively. Then we have
\begin{equation*}
N^{\spQ_\param}_{i,p}(x) = \phi_p^{\spQ_{\fracparam}}(nx-i+p+1),
\quad i=p+1,\ldots,n, \quad p\geq 2.
\end{equation*}
\end{remark}
This equality does not hold for $p=1$ because of the integral normalization factor $\delta_{1}^{U,V}$ in (\ref{eq:cardinal_GB_1}). On the other hand, this normalization factor ensures that the convolution relation (\ref{eq:convolution}) holds for all cardinal GB-splines.

\section{Symbols} \label{sec:symbols} 

Let $ \phi_p^{U,V} $ be the generalized cardinal B-splines as defined in Definition \ref{GB-splines-car}. In this Section we define and analyze three functions associated with certain Toeplitz matrices of interest later on. 

\subsection{Definitions and equivalent forms} \label{sec:symbols_def}

By denoting with $ \dot{\phi}_p^{U,V}(t) $ and $ \ddot{\phi}_p^{U,V}(t) $ the first and second derivative of $ \phi_p^{U,V}(t) $ with respect to its argument $t$, we can define, for $ \theta \in [-\pi,\pi] $:

\begin{align}
h_p^{U,V}(\theta) & := \phi_p^{U,V} \left( \frac{p+1}{2} \right)
                       + 2 \sum_{k=1}^{\lfloor p/2 \rfloor} 
                         \phi_p^{U,V} \left( \frac{p+1}{2}-k \right) \cos(k\theta),
                         \qquad p \geq 1 \label{c_eq:hp_finitesum} \\
g_p^{U,V}(\theta) & := - 2 \sum_{k=1}^{\lfloor p/2 \rfloor}
                         \dot{\phi}_p^{U,V} \left( \frac{p+1}{2}-k \right) \sin(k\theta),
                         \hspace{2.80cm} p \geq 2 \label{c_eq:gp_finitesum} \\
f_p^{U,V}(\theta) & := - \ddot{\phi}_p^{U,V} \left( \frac{p+1}{2} \right)
                       - 2 \sum_{k=1}^{\lfloor p/2 \rfloor} 
                         \ddot{\phi}_p^{U,V} \left( \frac{p+1}{2}-k \right) \cos(k\theta),
                         \hspace{0.40cm} p \geq 2 \label{c_eq:fp_finitesum}
\end{align}
%
Under the assumption \eqref{eq:sym_sym}, from the symmetry properties of cardinal GB-splines, the above functions also possess the equivalent form
\begin{align}
h_p^{U,V}:[-\pi, \pi]\rightarrow \RR, \quad \label{c_eq:hp}
h_p^{U,V}(\theta) & = \sum_{k\in\ZZ} {\phi}_{p}^{U,V}\left(\frac{p+1}{2}-k\right)\,\e^{\ii k \theta},
\\
g_p^{U,V}:[-\pi, \pi]\rightarrow \RR, \quad \label{c_eq:gp}
g_p^{U,V}(\theta) & = \ii \sum_{k\in\ZZ} \dot{\phi}_{p}^{U,V}\left(\frac{p+1}{2}-k\right)\,\e^{\ii k \theta},
\\
f_p^{U,V}:[-\pi, \pi]\rightarrow \RR, \quad \label{c_eq:fp}
f_p^{U,V}(\theta) & = - \sum_{k\in\ZZ} \ddot{\phi}_{p}^{U,V}\left(\frac{p+1}{2}-k\right)\,\e^{\ii k \theta}.
\end{align}

We are now looking for an alternative expression of $h_p^{U,V}$.
We first recall the Parseval identity for Fourier transforms, i.e.,
\begin{equation} \label{parseval}
  \int_\RR \varphi(t)\overline{\psi(t)}\,{\rm d}t = \frac{1}{2\pi} \int_\RR \widehat\varphi(\theta)\overline{\widehat\psi(\theta)}\,{\rm d}\theta,\quad\varphi,\,\psi\in L_2(\RR),
\end{equation}
and the translation property of the Fourier transform, i.e.,
\begin{equation} \label{shift}
 \widehat{\psi(\cdot+x)}(\theta) = \widehat\psi(\theta)\,\e^{\ii x\theta},\quad\psi\in L_1(\RR),\ x\in\RR.
\end{equation}
We also know that the Fourier transform of the cardinal GB-spline $ \phi_{p}^{U,V} $ 
can be written as
\begin{equation}
\label{c_eq:fourier-cardinal-GB-p}
\widehat{\phi_{p}^{U,V}}(\theta) 
= \widehat{\phi_{p-1}^{U,V}}(\theta)\, \widehat{\phi_{0}}(\theta)
= \widehat{\phi_{1}^{U,V}}(\theta)\, \bigl(\widehat{\phi_{0}}(\theta)\bigr)^{p-1}
=\widehat{\phi_{1}^{U,V}}(\theta)\,\left( \frac{1-\e^{-\ii \theta}}{\ii \theta} \right)^{p-1}.
\end{equation}
where explicit forms for the hyperbolic, trigonometric and polynomial cases can be given as:
\begin{align}
    \widehat{\phi_{1}^{\spH_\param}}(\theta)
& = \left(\frac{\param^2}{\cosh(\param)-1}\right)
    \left(\frac{\cosh(\param)-\cos(\theta)}{\theta^2+\param^2}\right)\e^{-\ii \theta},
    \quad \param\in\RR, \label{eq:fourier-cardinal-hyp-1} \\
    \widehat{\phi_{1}^{\spT_\param}}(\theta)
& = \left(\frac{\param^2}{1-\cos(\param)}\right)
    \left(\frac{\cos(\param)-\cos(\theta)}{\theta^2-\param^2}\right)\e^{-\ii \theta},
    \quad \param\in(0,\pi), \label{eq:fourier-cardinal-trig-1} \\
    \widehat{\phi_{1}}(\theta) 
& = \left(\frac{1-\e^{-\ii \theta}}{\ii \theta}\right)^2
  = 2\left(\frac{1-\cos(\theta)}{\theta^2}\right)\e^{-\ii \theta}
  = \lim_{\param\rightarrow 0}\widehat{\phi_{1}^{\spQ_\param}}(\theta). \label{eq:lim-trasf_phi1}
\end{align}
Then, by the convolution relation of cardinal GB-splines and by the symmetry of $\phi_{0}$, we get for $\ell\in\ZZ$,
\begin{align*}
&{\phi}_{p}^{U,V}\left(\frac{p+1}{2}-\ell\right)
  = \int_{\RR} \phi_{p-1}^{U,V}\left(\frac{p+1}{2}-\ell-t\right) \phi_{0}(t) \,{\rm d}t\\
  &\qquad= \int_{\RR} \phi_{p-1}^{U,V}(t)\, \phi_{0}\left(\frac{p+1}{2}-\ell-t\right) \,{\rm d}t
  = \int_{\RR} \phi_{p-1}^{U,V}(t)\, \phi_{0}\left(t+\ell+\frac{1-p}{2}\right) \,{\rm d}t.
\end{align*}
Applying (\ref{parseval}), (\ref{shift}) and (\ref{c_eq:fourier-cardinal-GB-p}) results in
\begin{align*}
&{\phi}_{p}^{U,V}\left(\frac{p+1}{2}-\ell\right)
  = \frac{1}{2\pi}  \int_{\RR} \widehat{\phi_{p-1}^{U,V}}(\theta)\, \overline{\widehat{\phi_{0}}(\theta)}\, \e^{-\ii (\ell+(1-p)/2) \theta}\,{\rm d}\theta\\
  &\qquad= \frac{1}{2\pi}  \int_{\RR} \widehat{\phi_{p-2}^{U,V}}(\theta)\, \bigl|{\widehat{\phi_{0}}(\theta)}\bigr|^2\, \e^{-\ii (\ell+(1-p)/2) \theta}\,{\rm d}\theta\\
  &\qquad= \frac{1}{2\pi} \sum_{k\in\ZZ} \int_{-\pi}^{\pi} \widehat{\phi_{p-2}^{U,V}}(\theta+2k\pi)\, \bigl|{\widehat{\phi_{0}}(\theta+2k\pi)}\bigr|^2\,(-1)^{k(p-1)} \e^{-\ii (\ell+(1-p)/2) \theta}\,{\rm d}\theta\\
  &\qquad= \frac{1}{2\pi} \int_{-\pi}^{\pi} \left[\sum_{k\in\ZZ} \widehat{\phi_{p-2}^{U,V}}(\theta+2k\pi)\, \bigl|{\widehat{\phi_{0}}(\theta+2k\pi)}\bigr|^2\,(-1)^{k(p-1)} \e^{\ii (p-1) \theta/2}\right]\e^{-\ii \ell\theta}\,{\rm d}\theta.
\end{align*}
We conclude that the values ${\phi}_{p}^{U,V}\left(\frac{p+1}{2}-\ell\right)$, $\ell\in\ZZ$ are the Fourier coefficients of both $h_p^{U,V}$ in (\ref{c_eq:hp}) and the above function between square brackets. Therefore,
$$
h_p^{U,V}(\theta) = \sum_{k\in\ZZ} \widehat{\phi_{p-2}^{U,V}}(\theta+2k\pi)\, \bigl|{\widehat{\phi_{0}}(\theta+2k\pi)}\bigr|^2\,(-1)^{k(p-1)} \e^{\ii (p-1) \theta/2}.
$$
Moreover, we have
$$
\bigl|{\widehat{\phi_{0}}(\theta+2k\pi)}\bigr|^2 = \frac{2-2\cos(\theta+2k\pi)}{(\theta+2k\pi)^2}
= \left(\frac{\sin(\theta/2+k\pi)}{\theta/2+k\pi}\right)^2,
$$
and
\begin{align*}
&\widehat{\phi_{p-2}^{U,V}}(\theta+2k\pi)(-1)^{k(p-1)} \e^{\ii (p-1) \theta/2} \\
&\qquad= \widehat{\phi_{1}^{U,V}}(\theta+2k\pi)\,\left( \frac{1-\e^{-\ii (\theta+2k\pi)}}{\ii (\theta+2k\pi)} \right)^{p-3} \e^{\ii (p-1) (\theta+2k\pi)/2}\\
&\qquad= \widehat{\phi_{1}^{U,V}}(\theta+2k\pi)\,\left( \frac{\sin(\theta/2+k\pi)}{\theta/2+k\pi} \right)^{p-3}\e^{\ii \theta}.
\end{align*}
This gives, for $ p \geq 3 $ 
$$
h_p^{U,V}(\theta) = \sum_{k\in\ZZ} \widehat{\phi_{1}^{U,V}}(\theta+2k\pi)\,\left( \frac{\sin(\theta/2+k\pi)}{\theta/2+k\pi} \right)^{p-1}\e^{\ii \theta}.
$$
In particular, using the explicit expressions for $\widehat{\phi_{1}}$, $\widehat{\phi_{1}^{\spH_\param}}$ and $\widehat{\phi_{1}^{\spT_\param}}$, as presented in \eqref{eq:fourier-cardinal-hyp-1}-\eqref{eq:lim-trasf_phi1}, we have 
\begin{align}
h_p(\theta) &= \sum_{k\in\ZZ} \left( \frac{\sin(\theta/2+k\pi)}{\theta/2+k\pi} \right)^{p+1}, \label{c_hpP} \\
h_p^{\spH_\param}(\theta) &= \sum_{k\in\ZZ} \left(\frac{\param^2}{\cosh(\param)-1}\right)
\left(\frac{\cosh(\param)-\cos(\theta)}{(\theta+2k\pi)^2+\param^2}\right)
\left( \frac{\sin(\theta/2+k\pi)}{\theta/2+k\pi} \right)^{p-1}, \label{c_hpH} \\
h_p^{\spT_\param}(\theta) &= \sum_{k\in\ZZ} \left(\frac{\param^2}{1-\cos(\param)}\right)
\left(\frac{\cos(\param)-\cos(\theta)}{(\theta+2k\pi)^2-\param^2}\right)
\left( \frac{\sin(\theta/2+k\pi)}{\theta/2+k\pi} \right)^{p-1}. \label{c_hpT}
\end{align}

%
\noindent Analogously, by using the differentiation property of cardinal GB-splines, we obtain:
\begin{align*}
&    \dot{\phi}_{p}^{U,V}\left(\frac{p+1}{2}-\ell\right) 
  = {\phi}_{p-1}^{U,V}\left(\frac{p+1}{2}-\ell\right) - {\phi}_{p-1}^{U,V}\left(\frac{p+1}{2}-\ell-1\right)\\
& =   \int_{\RR} {\phi}_{p-2}^{U,V} \left(\frac{p+1}{2}-\ell-t\right) \phi_0(t) {\rm d}t
    - \int_{\RR} {\phi}_{p-2}^{U,V} \left(\frac{p+1}{2}-\ell-1-t\right) \phi_0(t) {\rm d}t \\
& =   \int_{\RR} {\phi}_{p-2}^{U,V}(t) \phi_0\left(t+\ell+\frac{1-p}{2}\right) {\rm d}t
    - \int_{\RR} {\phi}_{p-2}^{U,V}(t) \phi_0\left(t+\ell+1+\frac{1-p}{2}\right) {\rm d}t \\
& = \frac{1}{2\pi} \int_{\RR} \widehat{\phi_{p-2}^{U,V}}(\theta) \overline{\hat{\phi}_0(\theta)}
    \left( \e^{-\ii(\ell+(1-p)/2)\theta} - \e^{-\ii(\ell+1+(1-p)/2)\theta} \right) {\rm d}\theta \\
& = \frac{1}{2\pi} \int_{\RR} \widehat{\phi_{p-3}^{U,V}}(\theta) \vert \hat{\phi}_0(\theta) \vert^2
    \left( \e^{-\ii(\ell+(1-p)/2)\theta} - \e^{-\ii(\ell+1+(1-p)/2)\theta} \right) {\rm d}\theta \\
& = \frac{1}{2\pi} \sum_{k \in \ZZ} \int_{-\pi}^{\pi} \widehat{\phi_{p-3}^{U,V}}(\theta+2k\pi) 
    \vert \hat{\phi}_0(\theta+2k\pi) \vert^2 (-1)^{k(p-1)} 
    \left( \e^{-\ii(\ell+(1-p)/2)\theta} - \e^{-\ii(\ell+1+(1-p)/2)\theta} \right) {\rm d}\theta \\
& = \frac{1}{2\pi} \sum_{k \in \ZZ} \int_{-\pi}^{\pi} \left[ \widehat{\phi_{p-3}^{U,V}}(\theta+2k\pi) 
    \vert \hat{\phi}_0(\theta+2k\pi) \vert^2 (-1)^{k(p-1)} 
    \left( \e^{\ii(p-1)\theta/2} - \e^{\ii(p-3)\theta/2} \right) \right] \e^{-\ii \ell\theta} {\rm d}\theta
\end{align*}
We conclude that the values $\dot{\phi}_{p}^{U,V}\left(\frac{p+1}{2}-\ell\right)$, $\ell\in\ZZ$, are the Fourier coefficients of both $g_p^{U,V}$ in (\ref{c_eq:gp}) and the above function between square brackets. Therefore,
$$ g_p^{U,V}(\theta) = \ii \sum_{k\in\ZZ} \widehat{\phi_{p-3}^{U,V}}(\theta+2k\pi)\, \bigl|{\widehat{\phi_{0}}(\theta+2k\pi)}\bigr|^2\,(-1)^{k(p-1)} \left( \e^{\ii(p-1)\theta/2} - \e^{\ii(p-3)\theta/2} \right), $$
which gives:
\begin{multline*}
   g_p^{\spT_\param}(\theta) = \ii \sum_{k\in\ZZ} \left(\frac{\param^2}{1-\cos(\param)}\right)
   \left(\frac{\cos(\param)-\cos(\theta)}{(\theta+2k\pi)^2-\param^2}\right) \e^{-\ii\theta}
   \left( \e^{\ii(p-1)\theta/2} - \e^{\ii(p-3)\theta/2} \right) (-1)^{k(p-1)} \\
   \left( \frac{1 - \e^{-\ii\theta}}{\ii (\theta + 2k\pi)} \right)^{p-4}
   \left( \frac{\sin(\theta/2+k\pi)}{\theta/2+k\pi} \right)^2.
\end{multline*}
Proceeding in the same way, we obtain also:
\begin{multline*}
   f_p^{\spT_\param}(\theta) = - \sum_{k\in\ZZ} \left(\frac{\param^2}{1-\cos(\param)}\right)
   \left(\frac{\cos(\param)-\cos(\theta)}{(\theta+2k\pi)^2-\param^2}\right) \e^{-\ii\theta} \\
   \left( \e^{\ii(p-1)\theta/2} - 2 \e^{\ii(p-3)\theta/2} + \e^{\ii(p-5)\theta/2} \right) (-1)^{k(p-1)}
   \left( \frac{1 - \e^{-\ii\theta}}{\ii (\theta + 2k\pi)} \right)^{p-5}
   \left( \frac{\sin(\theta/2+k\pi)}{\theta/2+k\pi} \right)^2.
\end{multline*} 
Alternative forms for $ g_p^{U,V}(\theta) $ and $ f_p^{U,V}(\theta) $ can be given by observing that:
\begin{align*}
&   \widehat{\phi_{p-3}^{U,V}}(\theta+2k\pi) (-1)^{k(p-1)} 
    \left( \e^{\ii(p-1)\theta/2} - \e^{\ii(p-3)\theta/2} \right) \\
& = \widehat{\phi_{p-3}^{U,V}}(\theta+2k\pi)
    \left( (-1)^{k(p-1)} \e^{\ii(p-1)\theta/2} - (-1)^{k(p-3)} \e^{\ii(p-3)\theta/2} \right) \\
& = \widehat{\phi_{p-3}^{U,V}}(\theta+2k\pi)
    \left( \e^{\ii(p-1)(\theta+2k\pi)/2} - \e^{\ii(p-3)(\theta+2k\pi)/2} \right) \\
& = \widehat{\phi_{p-3}^{U,V}}(\theta+2k\pi) \e^{\ii(p-2)(\theta+2k\pi)/2}
    \left( \e^{\ii(\theta+2k\pi)/2} - \e^{-\ii(\theta+2k\pi)/2} \right) \\
& = \widehat{\phi_1^{U,V}}(\theta+2k\pi) \left( \frac{\sin(\theta/2+k\pi)}{\theta/2+k\pi} \right)^{p-4}
    \e^{\ii\theta} \cdot 2\ii \cdot \sin(\theta/2+k\pi),
\end{align*}
and that:
\begin{align*}
&   \widehat{\phi_{p-4}^{U,V}}(\theta+2k\pi) (-1)^{k(p-1)} 
    \left( \e^{\ii(p-1)\theta/2} - 2 \e^{\ii(p-3)\theta/2} + \e^{\ii(p-5)\theta/2} \right) \\
& = \widehat{\phi_{p-4}^{U,V}}(\theta+2k\pi)
    \left( (-1)^{k(p-1)} \e^{\ii(p-1)\theta/2} - 2 (-1)^{k(p-3)} \e^{\ii(p-3)\theta/2} 
         + (-1)^{k(p-5)} \e^{\ii(p-5)\theta/2} \right) \\
& = \widehat{\phi_{p-4}^{U,V}}(\theta+2k\pi)
    \left( \e^{\ii(p-1)(\theta+2k\pi)/2} - 2 \e^{\ii(p-3)(\theta+2k\pi)/2} 
         + \e^{\ii(p-5)(\theta+2k\pi)/2} \right) \\
& = \widehat{\phi_{p-4}^{U,V}}(\theta+2k\pi) \e^{\ii(p-3)(\theta+2k\pi)/2}
    \left( \e^{2\ii(\theta+2k\pi)/2} - 2 + \e^{-2\ii(\theta+2k\pi)/2} \right) \\
& = \widehat{\phi_1^{U,V}}(\theta+2k\pi) \left( \frac{\sin(\theta/2+k\pi)}{\theta/2+k\pi} \right)^{p-5}
    \e^{\ii\theta} \cdot (-4) \cdot (\sin(\theta/2+k\pi))^2.
\end{align*}
This gives, respectively, for $ p \geq 4 $ in the case of $ g_p^{U,V}(\theta) $, and for $ p \geq 5 $ in the case of $ f_p^{U,V}(\theta) $ 
\begin{align*}
g_p^{U,V}(\theta) & = - 2 \sum_{k\in\ZZ} \widehat{\phi_{1}^{U,V}}(\theta+2k\pi)\,
                  \frac{{(\sin(\theta/2+k\pi))}^{p-1}}{{(\theta/2+k\pi)}^{p-2}} \cdot \e^{\ii \theta}, \\
f_p^{U,V}(\theta) & =   4 \sum_{k\in\ZZ} \widehat{\phi_{1}^{U,V}}(\theta+2k\pi)\,
                  \frac{{(\sin(\theta/2+k\pi))}^{p-1}}{{(\theta/2+k\pi)}^{p-3}} \cdot \e^{\ii \theta}.
\end{align*}
In particular, using the explicit expressions for $\widehat{\phi_{1}}$, $\widehat{\phi_{1}^{\spH_\param}}$ and $\widehat{\phi_{1}^{\spT_\param}}$, see again \eqref{eq:fourier-cardinal-hyp-1}-\eqref{eq:lim-trasf_phi1}, we have 
\begin{align}
g_p(\theta) &= -2 \sum_{k\in\ZZ} \frac{{(\sin(\theta/2+k\pi))}^{p+1}}{{(\theta/2+k\pi)}^p}, \notag \\
g_p^{\spH_\param}(\theta) &= -2 \sum_{k\in\ZZ} \left(\frac{\param^2}{\cosh(\param)-1}\right)
\left(\frac{\cosh(\param)-\cos(\theta)}{(\theta+2k\pi)^2+\param^2}\right)
\frac{{(\sin(\theta/2+k\pi))}^{p-1}}{{(\theta/2+k\pi)}^{p-2}}, \notag \\
g_p^{\spT_\param}(\theta) &= -2 \sum_{k\in\ZZ} \left(\frac{\param^2}{1-\cos(\param)}\right)
\left(\frac{\cos(\param)-\cos(\theta)}{(\theta+2k\pi)^2-\param^2}\right)
\frac{{(\sin(\theta/2+k\pi))}^{p-1}}{{(\theta/2+k\pi)}^{p-2}}, \notag \\
f_p(\theta) &= 4 \sum_{k\in\ZZ} \frac{{(\sin(\theta/2+k\pi))}^{p+1}}{{(\theta/2+k\pi)}^{p-1}}, 
\label{c_fpP} \\
f_p^{\spH_\param}(\theta) &= 4 \sum_{k\in\ZZ} \left(\frac{\param^2}{\cosh(\param)-1}\right)
\left(\frac{\cosh(\param)-\cos(\theta)}{(\theta+2k\pi)^2+\param^2}\right)
\frac{{(\sin(\theta/2+k\pi))}^{p-1}}{{(\theta/2+k\pi)}^{p-3}}, \label{c_fpH} \\
f_p^{\spT_\param}(\theta) &= 4 \sum_{k\in\ZZ} \left(\frac{\param^2}{1-\cos(\param)}\right)
\left(\frac{\cos(\param)-\cos(\theta)}{(\theta+2k\pi)^2-\param^2}\right)
\frac{{(\sin(\theta/2+k\pi))}^{p-1}}{{(\theta/2+k\pi)}^{p-3}}. \label{c_fpT}
\end{align}
%
%
For lower values of $p$, we have explicitly:
\begin{align}
    h_1^{\spT_\param}(\theta) & = (\param/2) \cot(\param/2) ; \qquad
    h_1^{\spH_\param}(\theta)   = (\param/2) \coth(\param/2) ; \qquad
    h_1(\theta) = 1; \label{c_expl:h1} \\
    h_2^{\spT_\param}(\theta)
& = \frac{\cos(\param/2)-1}{\cos(\param)-1} \cos(\theta) - \frac{\cos(\param/2)-\cos(\param)}{\cos(\param)-1};
    \label{c_expl:h2T} \\
    g_2^{\spT_\param}(\theta) & = \frac{\param \sin(\param/2)}{\cos(\param)-1} \sin(\theta) ; \qquad
    f_2^{\spT_\param}(\theta)   = \frac{\param^2 \cos(\param/2)}{1-\cos(\param)} (1-\cos(\theta)); 
    \label{c_expl:gf2T} \\
    h_2^{\spH_\param}(\theta)
& = \frac{\cosh(\param/2)-1}{\cosh(\param)-1} \cos(\theta) - \frac{\cosh(\param/2)-\cosh(\param)}
         {\cosh(\param)-1}; \label{c_expl:h2H} \\
    g_2^{\spH_\param}(\theta) & = - \frac{\param \sinh(\param/2)}{\cosh(\param)-1} \sin(\theta) ; \qquad
    f_2^{\spH_\param}(\theta)   = \frac{\param^2 \cosh(\param/2)}{\cosh(\param)-1} (1-\cos(\theta)); 
    \label{c_expl:gf2H} \\
    h_2(\theta) & = \frac{1}{4} \cos(\theta) + \frac{3}{4} ; \qquad g_2(\theta) = -\sin(\theta) ; \qquad
    f_2(\theta) = 2 - 2 \cos(\theta); \label{c_expl:hgf2P} \\
%
    g_3^{\spT_\param}(\theta) & = -\sin(\theta) ; \qquad
    f_3^{\spT_\param}(\theta)   = \param \cot(\param/2) (1-\cos(\theta)); \label{c_expl:gf3T} \\
    g_3^{\spH_\param}(\theta) & = -\sin(\theta) ; \qquad
    f_3^{\spH_\param}(\theta)   = \param \coth(\param/2) (1-\cos(\theta)); \label{c_expl:gf3H} \\
    g_3(\theta) & = -\sin(\theta) ; \qquad f_3(\theta) = 2 - 2 \cos(\theta); \label{c_expl:gf3P} \\
    f_4^{\spT_\param}(\theta) 
& = (2-2\cos(\theta)) \left[ \frac{\cos(\param/2)-1}{\cos(\param)-1} \cos(\theta) 
    - \frac{\cos(\param/2)-\cos(\param)}{\cos(\param)-1} \right]; \label{c_expl:f4T} \\
    f_4^{\spH_\param}(\theta)
& = (2-2\cos(\theta)) \left[ \frac{\cosh(\param/2)-1}{\cosh(\param)-1} \cos(\theta) 
    - \frac{\cosh(\param/2)-\cosh(\param)}{\cosh(\param)-1} \right]; \label{c_expl:f4H} \\
    f_4(\theta)
& = (2-2\cos(\theta)) \left[ \frac{1}{4} \cos(\theta) + \frac{3}{4} \right]. \label{c_expl:f4P}
\end{align}

\subsection{Bounds for \texorpdfstring{$ f_p^{\spQ_\param}(\theta) $}{fpQ} and \texorpdfstring{$ h_p^{\spQ_\param}(\theta) $}{hpQ}}

We look now for some lower and upper limitations concerning $ f_p^{\spQ_\param}(\theta) $ and $ h_p^{\spQ_\param}(\theta) $, the two symbols which are more important in the 1D setting. We do not consider, at least for the moment, an analogous study regarding $ g_p^{\spQ_\param}(\theta) $, because of the less marked interest in it while treating the unidimensional case. On the contrary, this is paired with an higher difficulty in obtaining noteworthy results. \\
We start by giving an equality which puts in relation $ f_p^{\spQ_\param}(\theta) $ and $ h_{p-2}^{\spQ_\param}(\theta) $.
\begin{lemma}
For both $ \spQ = \spH, \spT $, and for both $ p \geq 3 $ odd and even, we have
\begin{equation}
f_p^{\spQ_\param}(\theta) = (2 - 2 \cos (\theta))\, h_{p-2}^{\spQ_\param}(\theta). \label{c_rel_hpfp}
\end{equation}
\end{lemma}
\begin{proof}
For $p=3$, it descends from \eqref{c_expl:h1} and \eqref{c_expl:gf3T}-\eqref{c_expl:gf3H}; for $p=4$, it is true in light of \eqref{c_expl:h2T}, \eqref{c_expl:h2H}, and \eqref{c_expl:f4T}-\eqref{c_expl:f4H}; for $ p \geq 5 $, it is a consequence of \eqref{c_hpH}-\eqref{c_hpT} and \eqref{c_fpH}-\eqref{c_fpT}.
\end{proof}
Note that, as already stated in \cite[Eq. (3.13)]{our-MATHCOMP}, in the polynomial case we have as well
\begin{equation}
f_p(\theta) = (2 - 2 \cos (\theta))\, h_{p-2}(\theta), \label{c_rel_hpfp_pol}
\end{equation}
which derives also from \eqref{c_expl:h1} and \eqref{c_expl:gf3P} ($p=3$), \eqref{c_expl:hgf2P} and \eqref{c_expl:f4P} ($p=4$), \eqref{c_hpP} and \eqref{c_fpP} ($ p \geq 5 $). It is true also for $p=2$, with \eqref{c_expl:hgf2P} and since (only) in the polynomial case it is possible to define the symbol $h_p$ also for $p=0$, and precisely $h_0(\theta)=1$, see \cite[Eq. (3.10)]{our-MATHCOMP}. \\
For upper bounds, we have
\begin{lemma} \label{lem:upper_bound}
For both $ \spQ = \spH, \spT $, and for both $p$ odd and even, we have
\begin{align}
h_p^{\spQ_\param}(\theta) & \leq h_p^{\spQ_\param}(0) = 1, \quad p \geq 2, \label{c_max_hpQ} \\
f_p^{\spQ_\param}(\theta) & \leq 2 - 2 \cos(\theta),       \quad p \geq 4. \label{c_max_fpQ}
\end{align}
\end{lemma}
\begin{proof}
First, \eqref{c_max_hpQ} can be proved by means of \eqref{c_eq:hp}, thanks to 
$$ h_p^{\spQ_\param}(\theta) = \sum_{k\in\ZZ} {\phi}_{p}^{\spQ_\param}
                               \left(\frac{p+1}{2}-k\right)\,\e^{\ii k \theta}
                          \leq \sum_{k\in\ZZ} {\phi}_{p}^{\spQ_\param}
                               \left(\frac{p+1}{2}-k\right)\, \vert \e^{\ii k \theta} \vert = 1, $$
where it can be directly verified that $ h_p^{\spQ_\param}(0) = 1 $. \\
Then, \eqref{c_max_fpQ} is a consequence of \eqref{c_max_hpQ} and \eqref{c_rel_hpfp}.
\end{proof}

For lower bounds, we start by defining, for $ p \geq 3 $, a function $ r_p^{\spQ_\param}(\theta) $, as
\begin{align*}
r_p^{\spH_\param}(\theta) & = \left(\frac{\param^2}{\cosh(\param)-1}\right)
          \sum_{k \neq 0} \left(\frac{\cosh(\param)-\cos(\theta)}{(\theta+2k\pi)^2+\param^2}\right)
          \frac{(-1)^{k(p-1)}}{(\theta+2k\pi)^{p-1}}, \\
r_p^{\spT_\param}(\theta) & = \left(\frac{\param^2}{1-\cos(\param)}\right)
          \sum_{k \neq 0} \left(\frac{\cos(\param)-\cos(\theta)}{(\theta+2k\pi)^2-\param^2}\right) 
          \frac{(-1)^{k(p-1)}}{(\theta+2k\pi)^{p-1}}.
\end{align*}
Thanks to this definitions, we can operate the decompositions
\begin{align*}
h_p^{\spH_\param}(\theta) & = \left(\frac{\param^2}{\cosh(\param)-1}\right)
          \left(\frac{\cosh(\param)-\cos(\theta)}{\theta^2+\param^2}\right) 
          \left( \frac{\sin(\theta/2)}{\theta/2} \right)^{p-1}          
          + (2 \sin(\theta/2))^{p-1} r_p^{\spH_\param}(\theta), \\
h_p^{\spT_\param}(\theta) & = \left(\frac{\param^2}{1-\cos(\param)}\right)
          \left(\frac{\cos(\param)-\cos(\theta)}{\theta^2-\param^2}\right) 
          \left( \frac{\sin(\theta/2)}{\theta/2} \right)^{p-1}          
          + (2 \sin(\theta/2))^{p-1} r_p^{\spT_\param}(\theta).
\end{align*}
It is not easy to study $r_p^{\spQ_\param}(\theta)$ in its full generality, since the sign of the addenda composing it can be dependent of the various parameter $ \spQ, \alpha,k,p,\theta $. \\
Nevertheless, some results can be proved and some other ones can be reasonably conjectured, as follows:
\begin{lemma} \label{lem:lower_bound}
If $ \spQ = \spH $ and $p$ is odd, then $ (2 \sin(\theta/2))^{p-1} r_p^{\spH_\param}(\theta) \geq 0 $, and
\begin{align*}
h_p^{\spH_\param}(\theta) & \geq \left( \frac{\param^2}{\cosh(\param)-1} \right)
\left( \frac{\cosh(\param)+1}{\param^2+\pi^2} \right) \left( \frac{2}{\pi} \right)^{p-1},
\quad p \geq 1, \\
f_p^{\spH_\param}(\theta) & \geq (2-2\cos(\theta)) \left( \frac{\param^2}{\cosh(\param)-1} \right)
\left( \frac{\cosh(\param)+1}{\param^2+\pi^2} \right) \left( \frac{2}{\pi} \right)^{p-3},
\quad p \geq 3.
\end{align*}
In particular, $ f_p^{\spH_\param}(\theta) $ has a zero of multeplicity 2 for $\theta=0$, that is $ f_p^{\spH_\param}(0) = \displaystyle \frac{\partial f_p^{\spH_\param}}{\partial \theta}(0) = 0 $, and $ \displaystyle \frac{\partial^2 f_p^{\spH_\param}}{\partial \theta^2}(0) \neq 0 $. \\
\end{lemma}
Note that \eqref{c_expl:h1} and \eqref{c_expl:gf3T}-\eqref{c_expl:gf3H} (or \eqref{c_rel_hpfp}) are directly used in the lowest degree case, because $r_p^{\spH_\param}(\theta)$ is defined only for $ p \geq 3 $.
\begin{conjecture} \label{conj:lower_bound}
For both $ \spQ = \spH, \spT $, and for both $p$ odd and even, we have
\begin{align*}
h_p^{\spH_\param}(\theta) & \geq \left( \frac{\param^2}{\cosh(\param)-1} \right)
\left( \frac{\cosh(\param)+1}{\param^2+\pi^2} \right) \left( \frac{2}{\pi} \right)^{p-1},
\quad p \geq 1, \\
h_p^{\spT_\param}(\theta) & \geq \left( \frac{\param^2}{1-\cos(\param)} \right)
\left( \frac{\cos(\param)+1}{\pi^2-\param^2} \right) \left( \frac{2}{\pi} \right)^{p-1},
\quad p \geq 1, \\
f_p^{\spH_\param}(\theta) & \geq (2-2\cos(\theta)) \left( \frac{\param^2}{\cosh(\param)-1} \right)
\left( \frac{\cosh(\param)+1}{\param^2+\pi^2} \right) \left( \frac{2}{\pi} \right)^{p-3},
\quad p \geq 3, \\
f_p^{\spT_\param}(\theta) & \geq (2-2\cos(\theta)) \left( \frac{\param^2}{1-\cos(\param)} \right)
\left( \frac{\cos(\param)+1}{\pi^2-\param^2} \right) \left( \frac{2}{\pi} \right)^{p-3},
\quad p \geq 3.
\end{align*}
In particular, $ f_p^{\spQ_\param}(\theta) $ has a zero of multeplicity 2 for $\theta=0$, that is $ f_p^{\spQ_\param}(0) = \displaystyle \frac{\partial f_p^{\spQ_\param}}{\partial \theta}(0) = 0 $, and $ \displaystyle \frac{\partial^2 f_p^{\spQ_\param}}{\partial \theta^2}(0) \neq 0 $. \\
This is strongly supported by a large number of numerical tests, which show the behavior described in Lemma \ref{lem:lower_bound}, and proved there for certain sets of parameters, to hold in a more general setting. 
\end{conjecture}
We conclude this Subsection with some results relative to the polynomial case which are analogous to Lemmas \ref{lem:upper_bound}-\ref{lem:lower_bound} (and Conjecture \ref{conj:lower_bound}). From \cite[Lemmas 3.4 and 3.6]{our-MATHCOMP} we have
\begin{align}
\left( \frac{2}{\pi} \right)^{p+1} \leq
\left( \frac{2 - 2\cos(\theta)}{\theta^2} \right)^{\frac{p+1}{2}} & \leq h_p(\theta) \leq h_p(0) = 1,
\quad p \geq 0, \\
(2 - 2\cos(\theta)) \left( \frac{2}{\pi} \right)^{p-1} \leq
\frac{(2 - 2\cos(\theta))^{\frac{p+1}{2}}}{\theta^{p-1}} & \leq f_p(\theta) \leq 2 - 2\cos(\theta),
\quad p \geq 2.
\end{align}
Note that stricter bounds are available also for the upper limitations, and it is possible to limit also $g_p(\theta)$, see also \cite[Lemma 3.5]{our-MATHCOMP}.

\begin{figure}
   \includegraphics[scale=0.23]{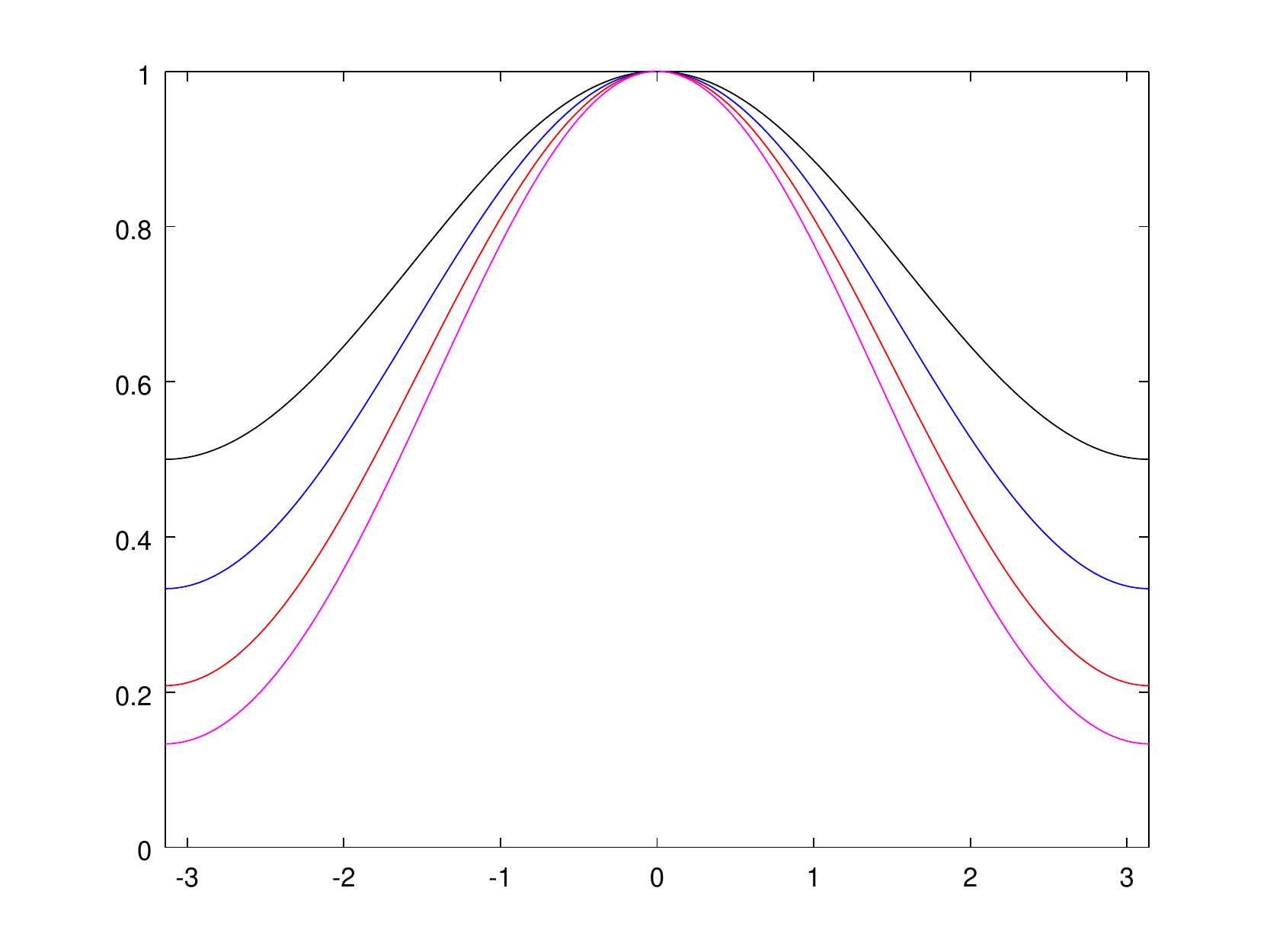} 
   \includegraphics[scale=0.23]{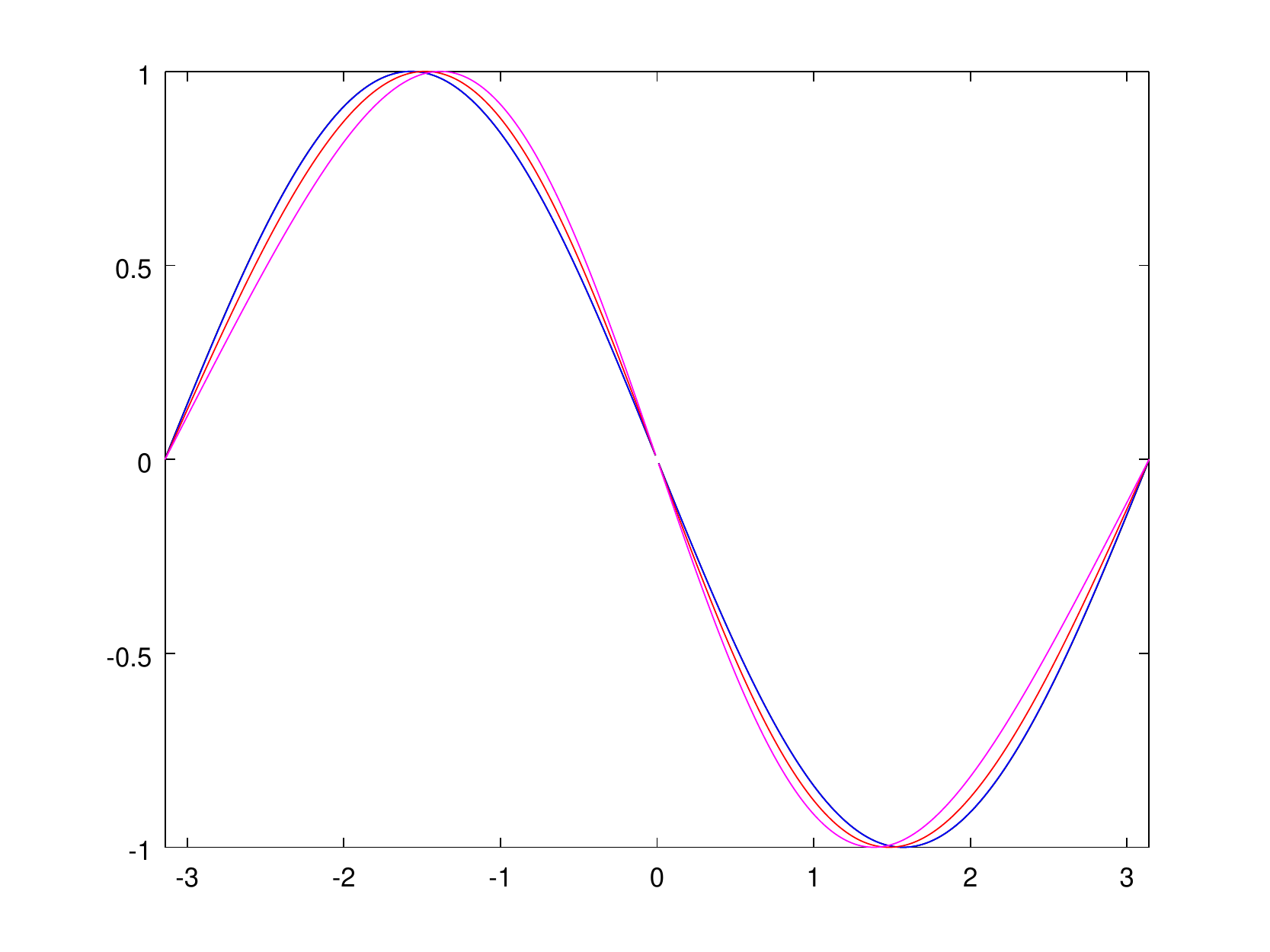}
   \includegraphics[scale=0.23]{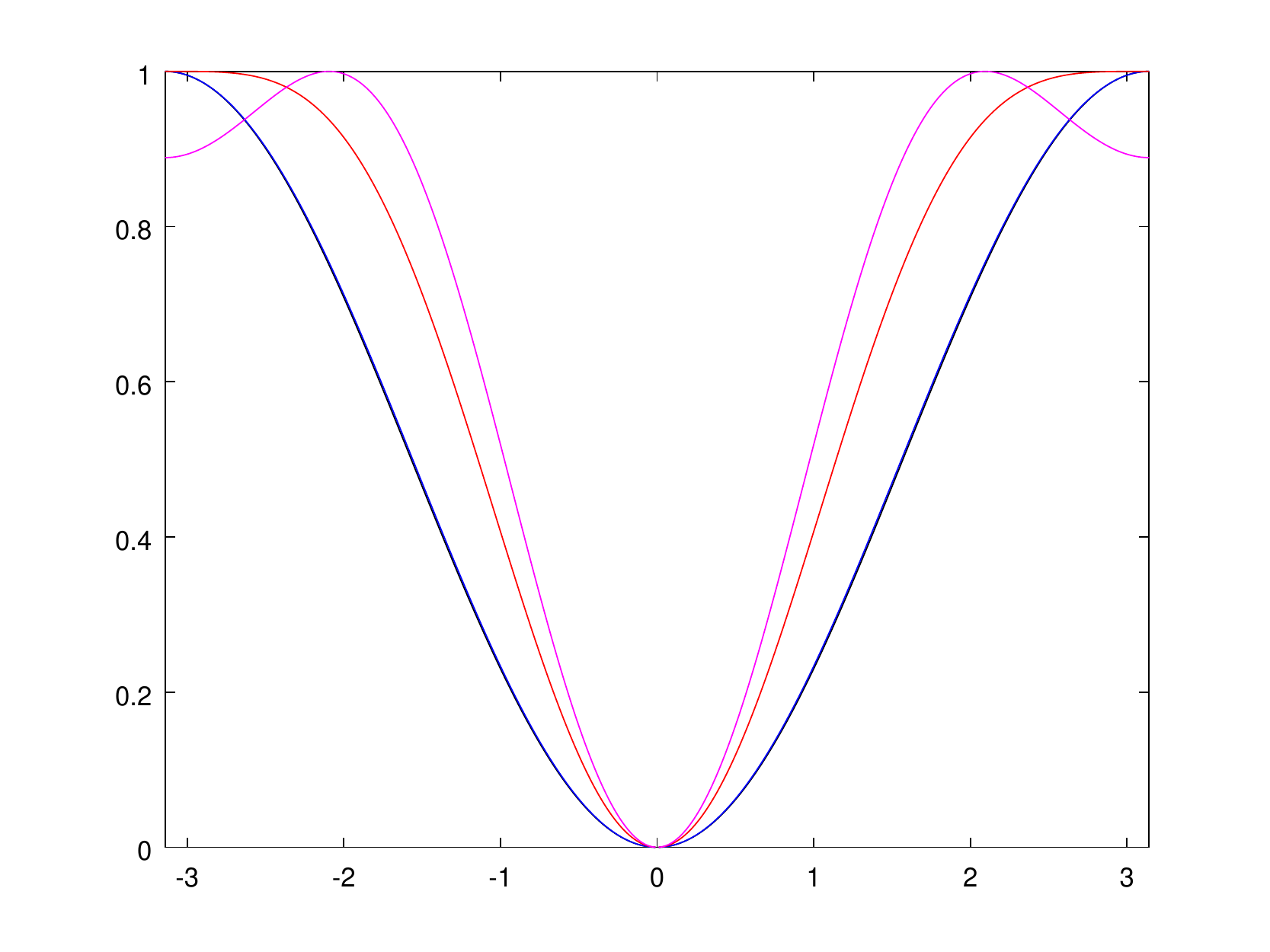}
\\ \includegraphics[scale=0.23]{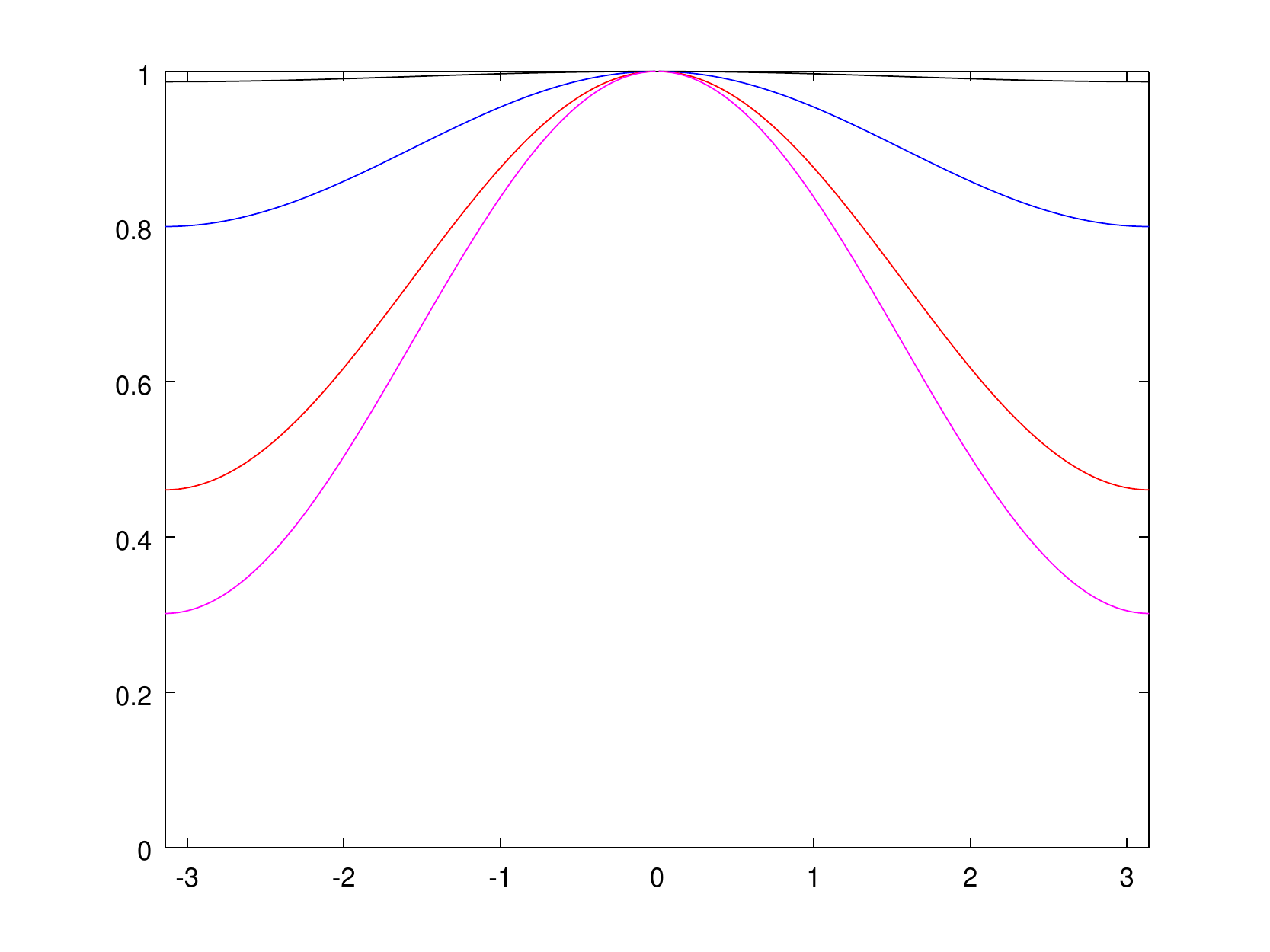} 
   \includegraphics[scale=0.23]{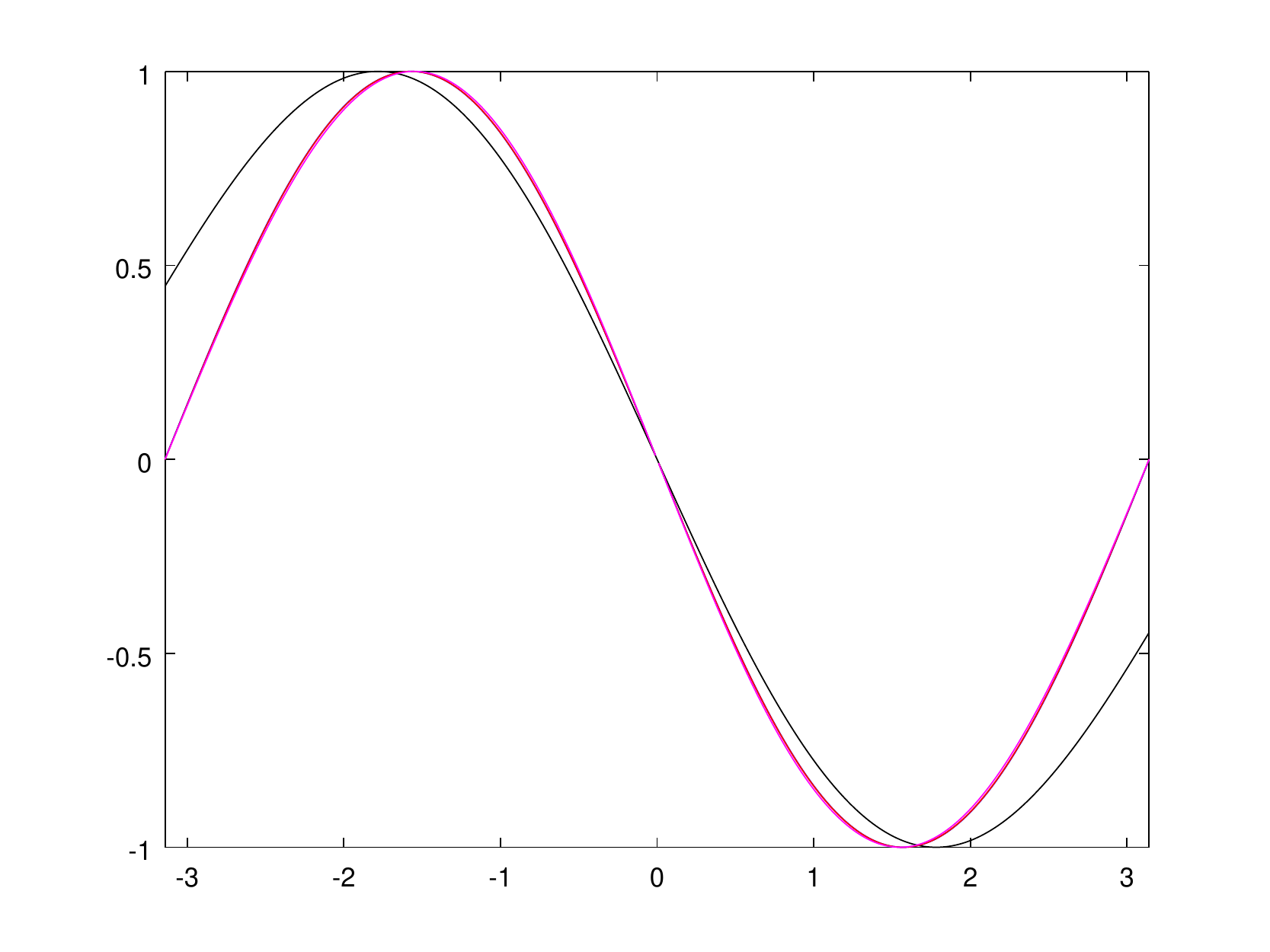}
   \includegraphics[scale=0.23]{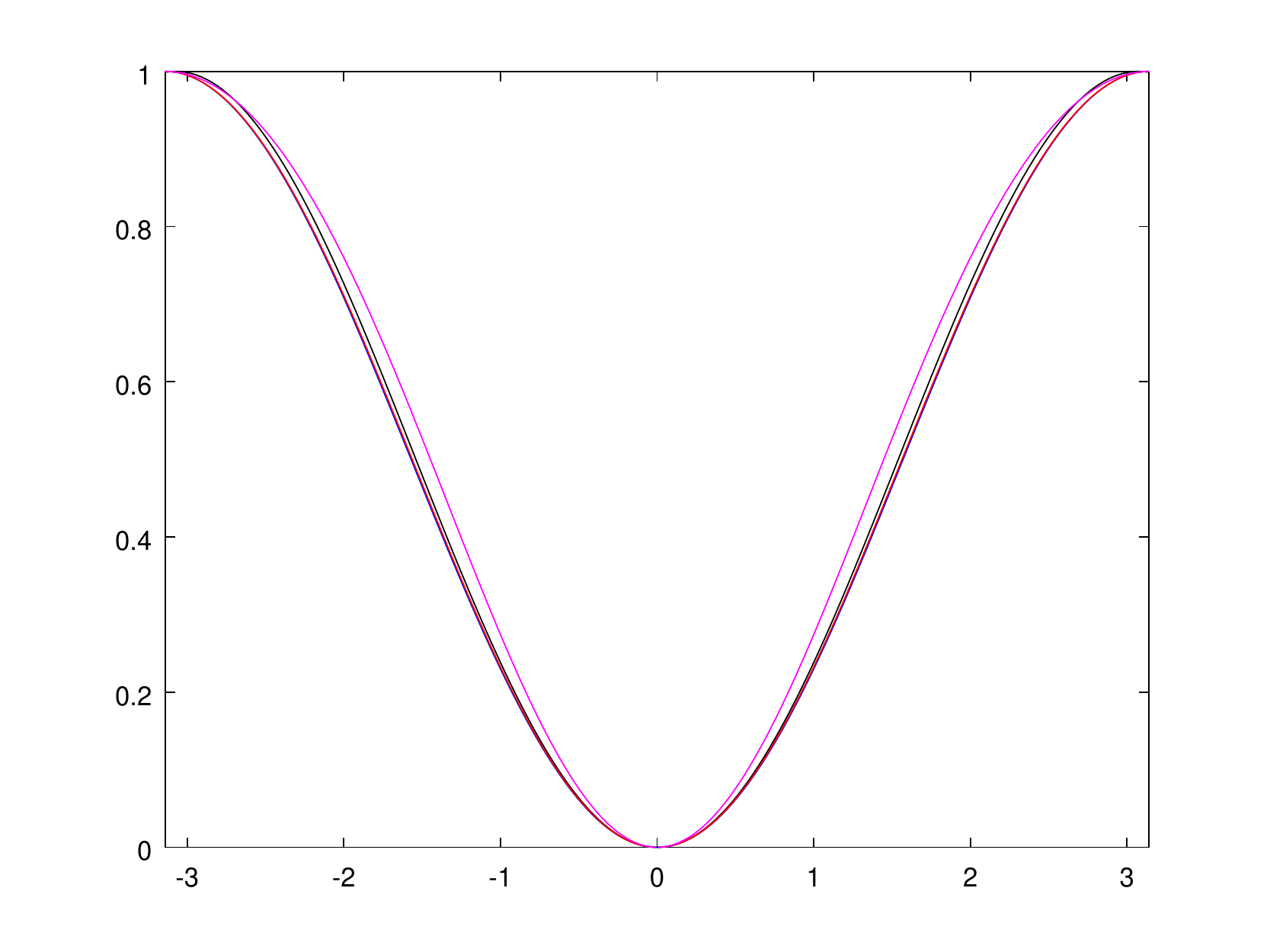}
\\ \includegraphics[scale=0.23]{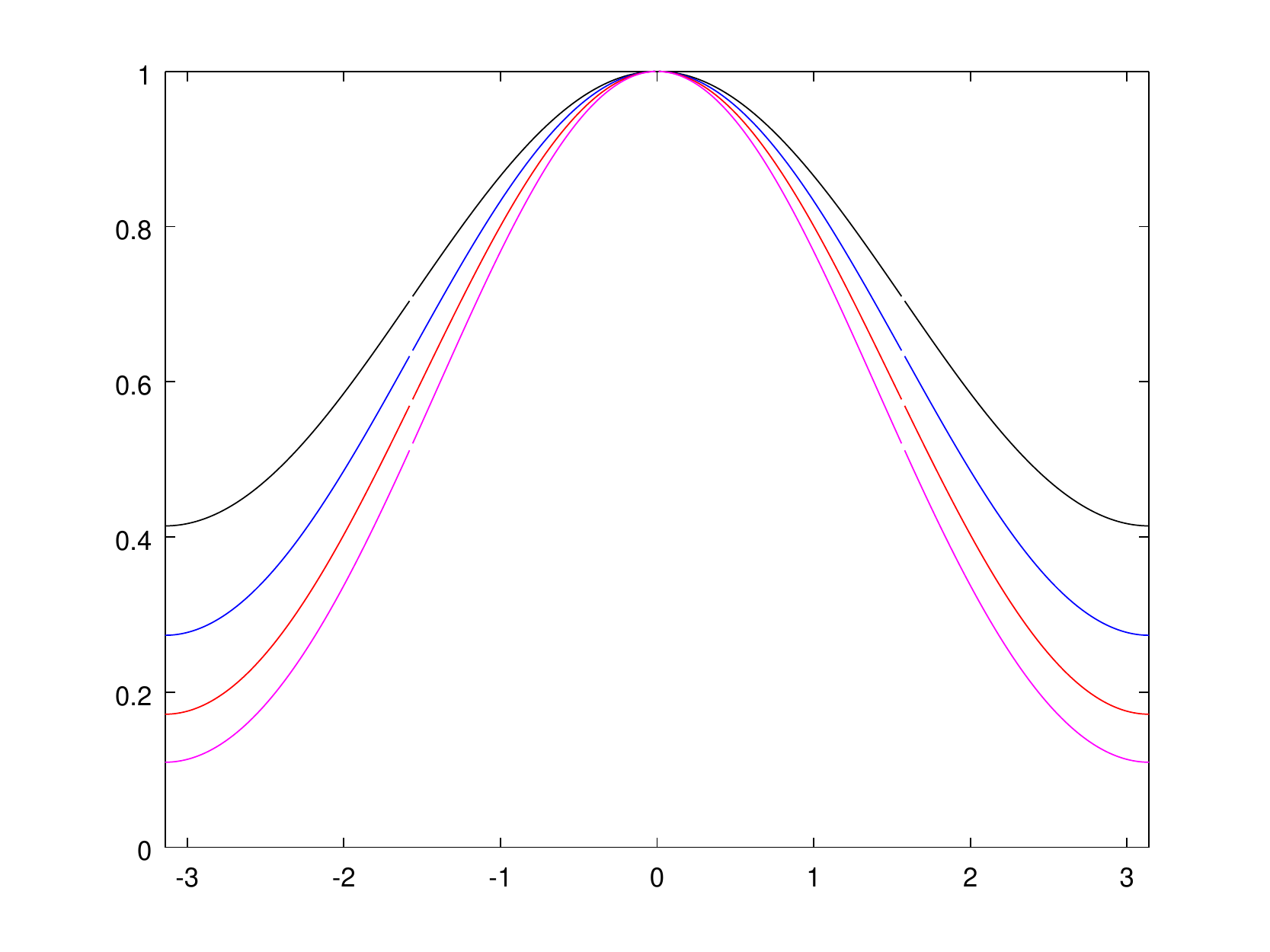} 
   \includegraphics[scale=0.23]{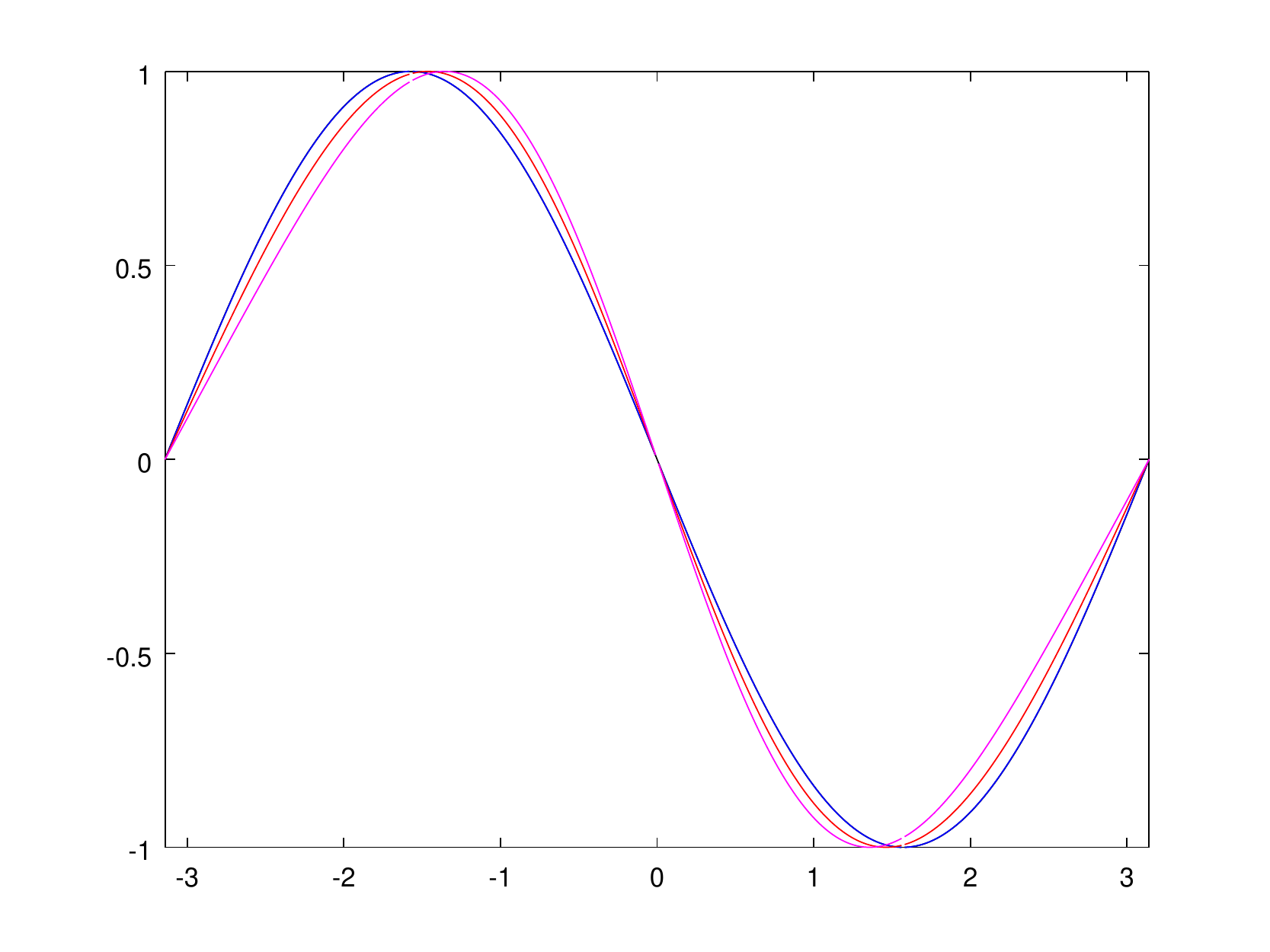}
   \includegraphics[scale=0.23]{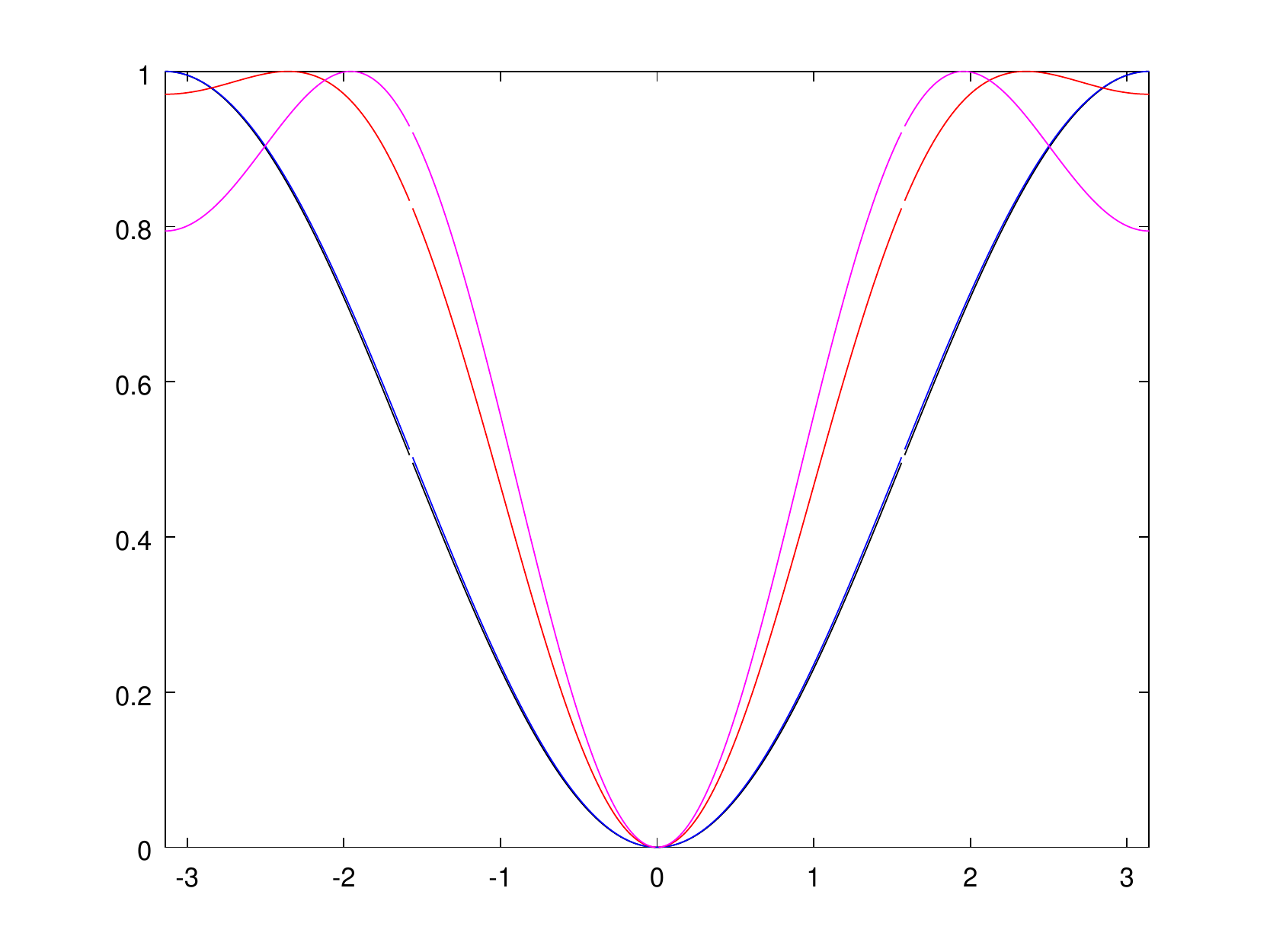}
\caption{From Top to Bottom. 
Left: plots of $h_p$, $h_p^{\spH_{10}}$, $h_p^{\spT_{\pi/2}}$.
Center. normalized plots of $g_p$, $g_p^{\spH_{10}}$, $g_p^{\spT_{\pi/2}}$
Right: normalized plots of $f_p$, $f_p^{\spH_{10}}$, $f_p^{\spT_{\pi/2}}$.
Black: $p=2$, blue: $p=3$, red: $p=4$, magenta: $p=5$.}
\label{c_fig:fp}
\end{figure}

\section{The 1D setting} \label{sec:col1D} 

We assume at first $ \Omega = (a,b) $. Then we focus on the problem

\begin{equation} \left\{ \begin{aligned} \label{c_eq:polycolloc401}
& -\kappa(x) u''(x) + \beta(x) u'(x) + \gamma(x) u(x) = \hbox{f}(x), \qquad a < x < b, \\
& u(a) = 0, \quad u(b) = 0
\end{aligned} \right. 
\end{equation}
We consider the approximation of the solution of \eqref{c_eq:polycolloc401} by the standard collocation approach, as explained briefly in the following. Let $ {\cal W} $ be a finite dimensional vector space of sufficiently smooth functions defined on $ \hbox{cl}(\Omega) = [a,b] $ and vanishing on the boundary $ \partial \Omega = \{a,b \} $. We call $ {\cal W} $ the approximation space. Then, we introduce a set of $ N := \dim {\cal W} $ collocation points in $ \Omega = (a,b) $,
\begin{equation} \label{c_eq:polycolloc202}
\{ \tau_i \in (a,b), \quad i = 1, \ldots, N \},
\end{equation}
and we look for a function $ u_{\cal W} \in {\cal W} $ such that
\begin{equation} \label{c_eq:polycolloc203}
- \kappa(\tau_i) u''_{\cal W}(\tau_i) + \beta(\tau_i) u'_{\cal W}(\tau_i) + \gamma(\tau_i) u_{\cal W}(\tau_i)
= \hbox{f}(\tau_i), \quad \forall \tau_i
\end{equation}
If we fix a basis $ \{ \varphi_1, \ldots, \varphi_N \} $ for $ {\cal W} $, then each $ v \in {\cal W} $ can be written as $ v = \sum_{j=1}^N v_j \varphi_j $. The collocation problem \eqref{c_eq:polycolloc203} is equivalent to the problem of finding a vector $ {\bf u} := [u_1 u_2 \cdots u_N]^T \in \RR^N $ such that
\begin{equation} \label{collocation_system}
A\bfu=\bff,
\end{equation}
where
\begin{equation*}
A := [- \kappa(\tau_i) \varphi''_j(\tau_i) + \beta(\tau_i) \varphi'_j(\tau_i) + \gamma(\tau_i) \varphi_j(\tau_i)]
     _{i,j=1}^N \in \RR^{N \times N}
\end{equation*}
is the collocation matrix and $ \bff := [\hbox{f}(\tau_i)]_{i=1}^N $. Once we find $\bfu$, we know $ u_{\cal W} = \sum_{j=1}^N u_j \varphi_j $. The regularity of the system \eqref{collocation_system} depends on the selection of the collocation points \eqref{c_eq:polycolloc202}. \\
We consider also a global geometry function $ G : [0,1] \rightarrow \hbox{cl}(\Omega) $, which we call \emph{geometry map}, describing the closure of the physical domain $ \hbox{cl}(\Omega) = [a,b] $ in terms of the parametric domain $[0,1]$. We assume that $G$ is invertible in $[0,1]$ and $ G(\partial([0,1])) = \partial \Omega = \{a,b \} $; this allows us to relate the respective basis functions as
\begin{equation} \label{c_eq:polycolloc207}
\varphi_i(x) := \hat{\varphi}_i(G^{-1}(x)) = \hat{\varphi}_i(\hat{x}), \quad x = G(\hat{x})
\end{equation}
and their collocation points as
\begin{equation} \label{c_eq:polycolloc208}
\tau_i := G(\hat{\tau}_i)
\end{equation}
With GB-splines, we suppose without loss of generality $ \Omega = (0,1) $, we fix $ p \geq 2, n \geq 2 $, and we let $ \mathcal{V}_{n,p}^{U,V} $ the space of GB-splines of degree $p$ defined over the knot sequence 
%
\begin{equation} \label{c_eq:poly34}
t_1 = \cdots = t_{p+1} = 0 < t_{p+2} < \cdots < t_{p+n} < 1 = t_{p+n+1} = \ldots = t_{2p+n+1},
\end{equation}
where
\begin{equation} \label{c_eq:poly35}
t_{p+i+1} := \frac{i}{n}, \qquad \forall i = 0, \ldots, n,
\end{equation}
and the extreme knots have multiplicity $p+1$. More precisely,
\begin{equation} \label{c_eq:poly35bis}
\mathcal{V}_{n,p}^{U,V} := \{ s \in C^{p-1}([0,1]) : s|_{[t_{p+i+1},t_{p+i+2})} \in \mathbb{P}_p^{U,V}, \quad
\forall i = 0, 1, \ldots, n-1 \}
\end{equation}
%
where $ \mathbb{P}_p^{U,V} $ is the section space \eqref{Puv}.
Let $ \mathcal{W}_{n,p}^{U,V} $ be the subspace of $ \mathcal{V}_{n,p}^{U,V} $ formed by the spline functions vanishing at the boundary of $[0,1]$, i.e.,
\begin{equation} \label{c_eq:poly36}
\mathcal{W}_{n,p}^{U,V} := \{ s \in \mathcal{V}_{n,p}^{U,V}: s(0)=s(1)=0 \} \subset H_0^1([0,1]).
\end{equation}
We recall that $ \dim \mathcal{V}_{n,p}^{U,V} = n+p $ and $ \dim \mathcal{W}_{n,p}^{U,V} = n+p-2 $. We choose the approximation space $ \mathcal{W} = \mathcal{W}_{n,p}^{U,V} $ for some $ p \geq 2, n \geq 2 $. \\
Note that, due to the fact the intervals have width of $1/n$, and not $1$, we have to consider a parameter change in the nonpolynomial functions, which is not necessary in the purely polynomial context, because there it does not correspond to a modification in the space. \\
Thus, we use the superscript $^{(U,V)(\cdot/n)}$ when we want to refer to an entity that depends of the functions $U$ and $V$, but with its parameter divided by $n$. \\
Indeed, the space (\ref{c_eq:poly35bis}) is spanned by a GB-spline basis $ \{ N^{U,V}_{i,p} : i = 1, \ldots, n+p \} $, for which the central basis functions, $ i = p+1, \ldots, n $, defined on the knot sequence (\ref{c_eq:poly34}) - (\ref{c_eq:poly35}), are cardinal GB-splines. Also, the space (\ref{c_eq:poly36}) is spanned by the same set, excluding the extremal splines $ N^{U,V}_{1,p} $ and $ N^{U,V}_{n+p,p} $. \\
An interesting set of collocation points for the space $ \mathcal{W}_{n,p}^{U,V} $ are the so-called Greville abscissae. The Greville abscissae corresponding to the GB-spline basis of $ \mathcal{V}_{n,p}^{U,V} $ are defined as
\begin{equation} \label{c_eq:polycolloc406}
\xi_{i,p} := \frac{t_{i+1} + t_{i+2} + \ldots + t_{i+p}}{p}, \quad i = 1, \ldots, n+p,
\end{equation}
so the set $ \{ \xi_{i+1,p}, i = 1, \ldots, n+p-2 \} $ can be taken as a set of collocation points for $ \mathcal{W}_{n,p}^{U,V} $. \\
%
The basis functions, by taking account also of the geometry map $G$, are given by \eqref{c_eq:polycolloc207} with
\begin{equation} \label{c_eq:polycolloc407}
\hat{\varphi}_i(\hat{x}) = N^{U,V}_{i+1,p}(\hat{x}), \quad i = 1, \ldots, n+p-2,
\end{equation}
and, see \eqref{c_eq:polycolloc208}
\begin{equation} \label{c_eq:polycolloc408}
\hat{\tau}_i = \xi_{i+1,p}, \quad i = 1, \ldots, n+p-2.
\end{equation}
The matrix $A$ in \eqref{collocation_system} based on \eqref{c_eq:polycolloc207} - \eqref{c_eq:polycolloc208} and \eqref{c_eq:polycolloc407} - \eqref{c_eq:polycolloc408} is the object of our interest; in the following, we will suppose for simplicity of notation that $G$ is the identity function, leaving to Sections \ref{geo_nonnested} and \ref{geo_nested} the case of a more general $G$. \\
Without a geometry map, the matrix $ A = A_{n,p}^{U,V} $ in \eqref{collocation_system} is equal to
\begin{align} \label{c_eq:polycolloc409}
A_{n,p}^{U,V} = [ - \kappa(\xi_{i+1,p}) (N^{U,V}_{j+1,p})'' (\xi_{i+1,p})
                & +  \beta(\xi_{i+1,p}) (N^{U,V}_{j+1,p})'  (\xi_{i+1,p}) \\
                & + \gamma(\xi_{i+1,p})  N^{U,V}_{j+1,p}    (\xi_{i+1,p})]_{i,j=1}^{n+p-2}. \notag
\end{align}
and $ {\bf f} = [ \hbox{f}(\xi_{i+1,p}) ]_{i=1}^{n+p-2} $. The following relationships hold
\begin{align}
N^{U,V}_{i,p}(x) & = \phi^{(U,V)(\cdot/n)}(nx-i+p+1), \quad i=p+1,\ldots,n, \label{c_eq:poly39} \\
(N^{U,V}_{i,p})'(x) & = n \dot{\phi}^{(U,V)(\cdot/n)}(nx-i+p+1), \quad i=p+1,\ldots,n, \label{c_eq:poly38bis} \\
(N^{U,V}_{i,p})''(x) & = n^2 \ddot{\phi}^{(U,V)(\cdot/n)}(nx-i+p+1), \quad i=p+1,\ldots,n. \label{c_eq:poly38ter}
\end{align}
In addition, the interior Greville abscissae, given by \eqref{c_eq:polycolloc406} for $ i = p+1, \ldots, n $, simplify to
\begin{equation} \label{c_eq:polycolloc411}
\xi_{i,p} = \frac{i}{n} - \frac{p+1}{2n}, \quad i = p+1, \ldots, n,
\end{equation}
or, equivalently,
\begin{equation*}
n \xi_{i,p} + p + 1 = i + \frac{p+1}{2}, \quad i = p+1, \ldots, n.
\end{equation*}
Let us denote by $ D_{n,p}^{U,V}(a) $ the diagonal matrix containing the samples of the function $a$ at the Greville abscissae, i.e.,
\begin{equation*}
D_{n,p}^{U,V}(a) := \hbox{diag}_{i=1, \ldots, n+p-2} (a(\xi_{i+1,p})).
\end{equation*}
Although $ D_{n,p}^{U,V}(a) $ does not actually depend on $U$ and $V$, we use this notation for the sake of similarity with the other matrices involved in the computations. \\
Then, we can consider the following split of $ A_{n,p}^{U,V} $,
\begin{equation} \label{c_eq:polycolloc412}
A_{n,p}^{U,V} =   n^2 D_{n,p}^{U,V}(\kappa) K_{n,p}^{U,V} + n D_{n,p}^{U,V}(\beta) H_{n,p}^{U,V}
                + D_{n,p}^{U,V}(\gamma) M_{n,p}^{U,V},
\end{equation}
%
according to the diffusion, advection and reaction terms, respectively. More precisely,
%
\begin{align}
n^2 K_{n,p}^{U,V} & := \left[ - (N^{U,V}_{j+1,p})'' (\xi_{i+1,p}) \right]_{i,j=1}^{n+p-2}, \label{c_eq:nnknp}
\\n H_{n,p}^{U,V} & := \left[   (N^{U,V}_{j+1,p})' (\xi_{i+1,p}) \right]_{i,j=1}^{n+p-2}, \label{c_eq:nhnp} 
\\  M_{n,p}^{U,V} & := \left[   N^{U,V}_{j+1,p} (\xi_{i+1,p}) \right]_{i,j=1}^{n+p-2}. \label{c_eq:mnp}
\end{align}
In a similar way with respect to the polynomial case, we can focus on the central submatrix of $ A_{n,p}^{U,V} $, which has again entries of indices between $p$ and $n-1$, being the Greville abscissae also exactly the same. Its splitted parts are:

\begin{align}
(K_{n,p}^{U,V})_{ij} & = -\ddot{\phi}_p^{(U,V)(\cdot/n)} \left( \frac{p+1}{2}+i-j \right) 
                       = T_{n-p}(f_p^{(U,V)(\cdot/n)}), \\
(H_{n,p}^{U,V})_{ij} & = \dot{\phi}_p^{(U,V)(\cdot/n)} \left( \frac{p+1}{2}+i-j \right) 
                       = \ii T_{n-p}(g_p^{(U,V)(\cdot/n)}), \\
(M_{n,p}^{U,V})_{ij} & = \phi_p^{(U,V)(\cdot/n)} \left( \frac{p+1}{2}+i-j \right) 
                       = T_{n-p}(h_p^{(U,V)(\cdot/n)}).
\end{align}
These central submatrices possess two structural properties which are crucial in the development of the theory:
\begin{itemize}
\item they are Toeplitz matrices, because their coefficients depend only on the difference $i-j$, and not on $i$ and $j$ alone;
\item they are either symmetric $ \left( \left[ (K_{n,p}^{U,V})_{ij} \right]_{i,j=p}^{n-1}, \left[ (M_{n,p}^{U,V})_{ij} \right]_{i,j=p}^{n-1} \right) $ or skew-symmetric \linebreak $ \left( \left[ (H_{n,p}^{U,V})_{ij} \right]_{i,j=p}^{n-1} \right) $, because swapping $i$ and $j$ means changing the difference $i-j$ by sign, and $\phi_p^{U,V}$, $\ddot{\phi}_p^{U,V}$ possess symmetry with respect to $\frac{p+1}{2}$ because $U$ and $V$ satisfy the hypotheses of \eqref{eq:sym_sym}, while $\dot{\phi}_p^{U,V}$ is antisymmetric w.r.t. the same value.
\end{itemize}
If we want to better characterize the central rows of the matrices $ K_{n,p}^{U,V}, H_{n,p}^{U,V}, M_{n,p}^{U,V} $, we can consider the row indices $ i \in \{ 1, \ldots, n+p-2 \} $ such that
\begin{align}
(K_{n,p}^{U,V})_{ij} & = -\ddot{\phi}_p^{(U,V)(\cdot/n)} \left( \frac{p+1}{2}+i-j \right), 
                         \quad \forall j = 1, \ldots, n+p-2, \label{c_eq:polycolloc419} \\
(H_{n,p}^{U,V})_{ij} & = \dot{\phi}_p^{(U,V)(\cdot/n)} \left( \frac{p+1}{2}+i-j \right),
                         \quad \forall j = 1, \ldots, n+p-2, \label{c_eq:polycolloc420} \\
(M_{n,p}^{U,V})_{ij} & = \phi_p^{(U,V)(\cdot/n)} \left( \frac{p+1}{2}+i-j \right),
                         \quad \forall j = 1, \ldots, n+p-2. \label{c_eq:polycolloc421}
\end{align}
The rows of $ K_{n,p}^{U,V}, H_{n,p}^{U,V} $ and $ M_{n,p}^{U,V} $ corresponding to indices $i$ satisfying \eqref{c_eq:polycolloc419} - \eqref{c_eq:polycolloc421} are called central rows. The cardinal GB-spline $ \phi_p^{U,V}$ is supported on the interval $ [0,p+1] $, and the support of any B-spline $ N^{U,V}_{j,p} $ is $ [t_j, t_{j+p+1}] $ with the knots given in \eqref{c_eq:poly34}-\eqref{c_eq:poly35}. Using these results and recalling that $ \xi_{i+1,p} = \frac{i+1}{n} - \frac{p+1}{2n} $ for $ i = p, \ldots, n-1 $, it can be shown that every $ i \in \{ \lfloor 3p/2 \rfloor, \ldots, n+p-1- \lfloor 3p/2 \rfloor \} $ satisfies \eqref{c_eq:polycolloc419} - \eqref{c_eq:polycolloc421}. Consequently, a condition to ensure that the matrices $ K_{n,p}^{U,V}, H_{n,p}^{U,V} $ and $ M_{n,p}^{U,V} $ have at least one central row is $ n \geq 2p + 1 - (p \mod 2) $. \\
We will now study, for a fixed $ p \geq 2 $, the spectral distribution of the sequence formed by the normalized matrices:
\begin{equation} \label{c_eq:polycolloc423}
\frac{1}{n^2} A_{n,p}^{U,V} = D_{n,p}^{U,V}(\kappa) K_{n,p}^{U,V} + \frac{1}{n} D_{n,p}^{U,V}(\beta) H_{n,p}^{U,V} 
                              + \frac{1}{n^2} D_{n,p}^{U,V}(\gamma) M_{n,p}^{U,V}.
\end{equation}
Let us decompose the matrix $ K_{n,p}^{U,V} $ into
\begin{equation} \label{c_eq:polycolloc424}
K_{n,p}^{U,V} = T_{n+p-2}(f_p^{(U,V)(\cdot/n)}) + R_{n,p}^{U,V}.
\end{equation}
From the construction of our diffusion matrix we know that 
$$ R_{n,p}^{U,V} := K_{n,p}^{U,V} - T_{n+p-2}(f_p^{(U,V)(\cdot/n)}) $$
is a low-rank correction term, which has at most $ 2 \lfloor 3p-2 \rfloor - 2 $ non-zero rows (having the central rows null), thus limiting the rank by that quantity. Similar decompositions gives also
\begin{equation} \label{c_eq:polycolloc426}
H_{n,p}^{U,V} = \ii T_{n+p-2}(g_p^{(U,V)(\cdot/n)}) + Q_{n,p}^{U,V}, \quad
M_{n,p}^{U,V} =     T_{n+p-2}(h_p^{(U,V)(\cdot/n)}) + S_{n,p}^{U,V}.
\end{equation}
From now onwards, we will only consider hyperbolic and trigonometric GB-spline discretizations because of their practical relevance. For the classical (polynomial) B-spline case we refer the reader to the spectral analysis in \cite{our-MATHCOMP}. 

As already described in Remark~\ref{rmk:def-agreement} and in \eqref{c_eq:poly39}-\eqref{c_eq:poly38ter}, for $\spQ=\spH, \spT$ the central basis functions $N^{\spQ_\paramgen}_{i,p}, i=p+1,\ldots,n$ can be expressed in terms of cardinal GB-splines, i.e., 
\begin{equation} \label{c_eq:poly39q}
N^{\spQ_\paramgen}_{i,p}(x) = \phi_p^{\spQ_{\fracparamgen}}(nx-i+p+1), \quad i=p+1,\ldots,n,
\quad p\geq2,
\end{equation}
and as a consequence,
\begin{align}
    \bigl(N^{\spQ_\paramgen}_{i,p}\bigr)'(x) 
& = n \dot{\phi}^{\spQ_{\fracparamgen}}_{p}(nx-i+p+1), \quad i=p+1,\ldots,n, \quad p\geq2, \label{c_eq:poly39bis} \\
    \bigl(N^{\spQ_\paramgen}_{i,p}\bigr)''(x) 
& = n^2 \ddot{\phi}^{\spQ_{\fracparamgen}}_{p}(nx-i+p+1), \quad i=p+1,\ldots,n, \quad p\geq2. \label{c_eq:poly39ter}
\end{align}

In view of (\ref{c_eq:poly39q}), the two choices for the parameter $\paramgen$ described in \cite[Section 4]{galerkin_gbsp} are again the ones of interest when considering a sequence of spaces $\WW_{n,p}^{\spQ_\paramgen}$, $n\rightarrow\infty$. They are: 
\begin{itemize}
\item $\paramgen=\param$ where $\param$ is a given real value.
This is the most natural choice for collocation IgA discretizations, because it generates a sequence of nested spaces $\WW_{n,p}^{\spQ_\param}$ given a sequence of nested knot sets.
The basis functions in (\ref{c_eq:poly39q}) are a scaled shifted version of a sequence of cardinal hyperbolic/trigonometric GB-splines identified by a sequence of parameters approaching $0$ as $n$ increases.
\item $\paramgen=n\param$ where $\param$ is a given real value.
In this case, the basis functions in (\ref{c_eq:poly39q}) are a scaled shifted version of the same cardinal hyperbolic/trigonometric GB-spline for all $n$. The corresponding spaces $\WW_{n,p}^{\spQ_{n\param}}$ are not nested.
Nevertheless, the analysis of the spectral properties of the sequence of matrices in this case gives some interesting insights into the spectrum of the corresponding matrices in the above (nested) case for finite $n$.
\end{itemize} 
As already done, we will address these two choices and will refer to them as the nested ($ \paramgen = \param $) and non-nested case ($ \paramgen = n\param $). We recall that the parameter $\mu$ is constrained by the Tchebycheff structure of the section spaces (see (\ref{Exp}) and (\ref{Trig})). \\
Just before starting with analyzing the consequence of every choice, we define the (1-level) diagonal sampling matrix $D_m(a)$ associated with a Riemann integrable function $ a : [0,1] \rightarrow \CC $, which writes as:
\begin{equation}
D_m(a) = \hbox{diag}_{i=0,\ldots,m-1} \left( a \left(\frac{i}{m} \right) \right),
\end{equation}
This matrix is associated with a less sparse form of it, which puts in a ledge structure its diagonal elements, and which will be used in some of the subsequent proofs.
\begin{equation}
(\tilde{D}_m(a))_{i,j} := (D_m(a))_{\min(i,j),\min(i,j)} =
\left\{ \begin{aligned}
& (D_m(a))_{i,i}, \quad \hbox{if } i \leq j \\
& (D_m(a))_{j,j}, \quad \hbox{if } i > j
\end{aligned} \right.
\end{equation}

\subsection{The non-nested case}

Some limitations on the norms of $ M_{n,p}^{\spQ_{n\param}}, H_{n,p}^{\spQ_{n\param}}, K_{n,p}^{\spQ_{n\param}} $ can be devised.
\begin{lemma} \label{lem:poly41}
For every $ p \geq 2 $ and every $ n \geq 2 $, we have
$$ \Vert M_{n,p}^{\spQ_{n\param}} \Vert_2 \leq \sqrt{\frac{3p}{2}}, \quad
   \Vert H_{n,p}^{\spQ_{n\param}} \Vert_2 \leq C_{p,\param}^{(1)}, \quad
   \Vert K_{n,p}^{\spQ_{n\param}} \Vert_2 \leq C_{p,\param}^{(2)} $$
where $ C_{p,\param}^{(1)} $ and $ C_{p,\param}^{(2)} $ are constants independent of $n$.
\end{lemma}
\begin{proof}
We recall that the $2$-norm of any square matrix $X$ can be bounded as
\begin{equation} \label{c_eq:norm2bound}
\Vert X \Vert_2 \leq \sqrt{\Vert X \Vert_{\infty} \Vert X^T \Vert_{\infty}}
\end{equation}
In light of \eqref{c_eq:norm2bound}, we can look for bounds of the infinity norm of the matrices $ M_{n,p}^{\spQ_{n\param}}, H_{n,p}^{\spQ_{n\param}}, K_{n,p}^{\spQ_{n\param}} $ and their transposes. \\
From \eqref{c_eq:mnp} we can obtain:
$$   \Vert M_{n,p}^{\spQ_{n\param}} \Vert_{\infty} 
   = \max_{i=1,\ldots,n+p-2} \sum_{j=1}^{n+p-2} N^{\spQ_{n\param}}_{j+1,p}(\xi_{i+1,p}) \leq 1 $$
because, for every $ i = 1, \ldots, n+p-2 $, if we consider the whole set of GB-splines whose support contains $ \xi_{i+1,p} $, namely $ \mathcal{G}_{i+1,p} $, then $ \sum_{k \in \mathcal{G}_{i+1,p}} N_{k,p}^{\spQ_{n\param}}(\xi_{i+1,p}) = 1 $, and, $ \forall k \in \mathcal{G}_{i+1,p}, N_{k,p}^{\spQ_{n\param}}(\xi_{i+1,p}) \geq 0 $, thanks to the positivity property and the partition of unity property. \\ 
On the other side,
$$   \Vert (M_{n,p}^{\spQ_{n\param}})^T \Vert_{\infty}
   = \max_{j=1,\ldots,n+p-2} \sum_{i=1}^{n+p-2} N^{\spQ_{n\param}}_{j+1,p}(\xi_{i+1,p}) \leq \frac{3p}{2}, $$
because the number of Greville abscissae in the interior of the support of any GB-spline is at most $\frac{3p}{2}$, and we use again the positivity property and the partition of unity property. \\
The bound for $ \Vert M_{n,p}^{\spQ_{n\param}} \Vert_2 $ follows from \eqref{c_eq:norm2bound}. \\
Similarly, from \eqref{c_eq:nhnp}, and by using the recurrence relation for derivatives, we have: 
\begin{align*}
       \Vert n H_{n,p}^{\spQ_{n\param}} \Vert_{\infty} 
&    = \max_{i=1,\ldots,n+p-2} \sum_{j=1}^{n+p-2} \vert (N^{\spQ_{n\param}}_{j+1,p})'(\xi_{i+1,p}) \vert \\
& \leq \max_{i=1,\ldots,n+p-2} \sum_{j=1}^{n+p-2}
       [   \delta_{j+1,p-1}^{\spQ_{n\param}} N_{j+1,p-1}^{\spQ_{n\param}}(\xi_{i+1,p})
         + \delta_{j+2,p-1}^{\spQ_{n\param}} N_{j+2,p-1}^{\spQ_{n\param}}(\xi_{i+1,p}) ] \\
& \leq \max_{i=1,\ldots,n+p-2} 
       \left( \max_{j=1,\ldots,n+p-2} \delta_{j+1,p-1}^{\spQ_{n\param}} \right.
       \left. + \max_{j=1,\ldots,n+p-2} \delta_{j+2,p-1}^{\spQ_{n\param}} \right) \\
&    =        \max_{j=1,\ldots,n+p-2} \delta_{j+1,p-1}^{\spQ_{n\param}}
              + \max_{j=1,\ldots,n+p-2} \delta_{j+2,p-1}^{\spQ_{n\param}} \leq 
              \left( \frac{3p}{2} \right)^{-\frac{1}{2}} C_{p,\param}^{(1)} n
\end{align*}
and:
\begin{align*}
       \Vert (n H_{n,p}^{\spQ_{n\param}})^T \Vert_{\infty} 
&    = \max_{j=1,\ldots,n+p-2} \sum_{i=1}^{n+p-2} \vert (N^{\spQ_{n\param}}_{j+1,p})'(\xi_{i+1,p}) \vert \\
& \leq \max_{j=1,\ldots,n+p-2} \sum_{i=1}^{n+p-2}
       [   \delta_{j+1,p-1}^{\spQ_{n\param}} N_{j+1,p-1}^{\spQ_{n\param}}(\xi_{i+1,p})
         + \delta_{j+2,p-1}^{\spQ_{n\param}} N_{j+2,p-1}^{\spQ_{n\param}}(\xi_{i+1,p}) ] \\
& \leq \max_{j=1,\ldots,n+p-2} \frac{3p}{2} [\delta_{j+1,p-1}^{\spQ_{n\param}} +
       \delta_{j+2,p-1}^{\spQ_{n\param}}] \\
& \leq \frac{3p}{2} \left( \max_{j=1,\ldots,n+p-2} \delta_{j+1,p-1}^{\spQ_{n\param}} \right.
       \left. + \max_{j=1,\ldots,n+p-2} \delta_{j+2,p-1}^{\spQ_{n\param}} \right) \leq 
       \sqrt{\frac{3p}{2}} C_{p,\param}^{(1)} n
\end{align*}
for a proper choice of the constant $ C_{p,\param}^{(1)} $. Again \eqref{c_eq:norm2bound} gives the bound for $ \Vert H_{n,p}^{\spQ_{n\param}} \Vert_2 $. \\ 
Finally, from \eqref{c_eq:nnknp}, and by reiterating \eqref{eq:der-recurrence}:
\begin{align*}
       \Vert n^2 K_{n,p}^{\spQ_{n\param}} \Vert_{\infty} 
&    = \max_{i=1,\ldots,n+p-2} \sum_{j=1}^{n+p-2} \vert (N^{\spQ_{n\param}}_{j+1,p})''(\xi_{i+1,p}) \vert \\
& \leq \max_{i=1,\ldots,n+p-2} \sum_{j=1}^{n+p-2} 
       \vert   \delta_{j+1,p-1}^{\spQ_{n\param}} (N^{\spQ_{n\param}}_{j+1,p-1})'(\xi_{i+1,p}) 
             + \delta_{j+2,p-1}^{\spQ_{n\param}} (N^{\spQ_{n\param}}_{j+2,p-1})'(\xi_{i+1,p}) \vert \\
& \leq \max_{i=1,\ldots,n+p-2} \sum_{j=1}^{n+p-2}
  \vert \delta_{j+1,p-1}^{\spQ_{n\param}} 
        \delta_{j+1,p-2}^{\spQ_{n\param}} N^{\spQ_{n\param}}_{j+1,p-2}(\xi_{i+1,p})
\\  & + \delta_{j+1,p-1}^{\spQ_{n\param}} 
        \delta_{j+2,p-2}^{\spQ_{n\param}} N^{\spQ_{n\param}}_{j+2,p-2}(\xi_{i+1,p})
      + \delta_{j+2,p-1}^{\spQ_{n\param}} 
        \delta_{j+2,p-2}^{\spQ_{n\param}} N^{\spQ_{n\param}}_{j+2,p-2}(\xi_{i+1,p})
\\  & + \delta_{j+2,p-1}^{\spQ_{n\param}} 
        \delta_{j+3,p-2}^{\spQ_{n\param}} N^{\spQ_{n\param}}_{j+3,p-2}(\xi_{i+1,p}) \vert
\\  & \leq \max_{i=1,\ldots,n+p-2} 
    \left( \max_{j=1,\ldots,n+p-2} \delta_{j+1,p-1}^{\spQ_{n\param}} \delta_{j+1,p-2}^{\spQ_{n\param}}
    + \max_{j=1,\ldots,n+p-2} \delta_{j+1,p-1}^{\spQ_{n\param}} \delta_{j+2,p-2}^{\spQ_{n\param}} \right. \\
    & \left. + \max_{j=1,\ldots,n+p-2} \delta_{j+2,p-1}^{\spQ_{n\param}} \delta_{j+2,p-2}^{\spQ_{n\param}}
    + \max_{j=1,\ldots,n+p-2} \delta_{j+2,p-1}^{\spQ_{n\param}} \delta_{j+3,p-2}^{\spQ_{n\param}} \right) \\
    & \leq \left( \frac{3p}{2} \right)^{-\frac{1}{2}} C_{p,\param}^{(2)} n^2
\end{align*}
and:
\begin{align*}
       \Vert (n^2 K_{n,p}^{\spQ_{n\param}})^T \Vert_{\infty} 
&    = \max_{j=1,\ldots,n+p-2} \sum_{i=1}^{n+p-2} \vert (N^{\spQ_{n\param}}_{j+1,p})''(\xi_{i+1,p}) \vert \\
& \leq \max_{j=1,\ldots,n+p-2} \sum_{i=1}^{n+p-2} 
       \vert   \delta_{j+1,p-1}^{\spQ_{n\param}} (N^{\spQ_{n\param}}_{j+1,p-1})'(\xi_{i+1,p}) 
             + \delta_{j+2,p-1}^{\spQ_{n\param}} (N^{\spQ_{n\param}}_{j+2,p-1})'(\xi_{i+1,p}) \vert \\
& \leq \max_{j=1,\ldots,n+p-2} \sum_{i=1}^{n+p-2}
  \vert \delta_{j+1,p-1}^{\spQ_{n\param}} 
        \delta_{j+1,p-2}^{\spQ_{n\param}} N^{\spQ_{n\param}}_{j+1,p-2}(\xi_{i+1,p})
\\ & +  \delta_{j+1,p-1}^{\spQ_{n\param}} 
        \delta_{j+2,p-2}^{\spQ_{n\param}} N^{\spQ_{n\param}}_{j+2,p-2}(\xi_{i+1,p})
     +  \delta_{j+2,p-1}^{\spQ_{n\param}} 
        \delta_{j+2,p-2}^{\spQ_{n\param}} N^{\spQ_{n\param}}_{j+2,p-2}(\xi_{i+1,p})
\\ & +  \delta_{j+2,p-1}^{\spQ_{n\param}} 
        \delta_{j+3,p-2}^{\spQ_{n\param}} N^{\spQ_{n\param}}_{j+3,p-2}(\xi_{i+1,p}) \vert
\\ & \leq \frac{3p}{2}
     \left( \max_{j=1,\ldots,n+p-2} \delta_{j+1,p-1}^{\spQ_{n\param}} \delta_{j+1,p-2}^{\spQ_{n\param}}
     + \max_{j=1,\ldots,n+p-2} \delta_{j+1,p-1}^{\spQ_{n\param}} \delta_{j+2,p-2}^{\spQ_{n\param}} \right. 
\\ & \left. + \max_{j=1,\ldots,n+p-2} \delta_{j+2,p-1}^{\spQ_{n\param}} \delta_{j+2,p-2}^{\spQ_{n\param}}
     + \max_{j=1,\ldots,n+p-2} \delta_{j+2,p-1}^{\spQ_{n\param}} \delta_{j+3,p-2}^{\spQ_{n\param}} \right) 
\\ & \leq  \sqrt{\frac{3p}{2}} C_{p,\param}^{(2)} n^2
\end{align*}
again for a proper choice of the constant $ C_{p,\param}^{(2)} $. Also \eqref{c_eq:norm2bound} gives the bound for $ \Vert K_{n,p}^{\spQ_{n\param}} \Vert_2 $.
\end{proof}
In the following, we denote by $C$ a generic constant independent of $n$, and by $ \omega(s,\delta) $ the modulus of continuity, that, in our case, reads as 
$$ \omega(s,\delta) := \sup_{x,y \in [0,1], \vert x-y \vert \leq \delta } \vert s(x) - s(y) \vert. $$
%
%
%
%
%
%
%
The main theorem for the non-nested case follows. Here and in the following, we use the notation $ M_{f_p^{\spQ_\param}} := \underset{\theta \in [-\pi,\pi]}{\max} f_p^{\spQ_\param}(\theta) $ and $ M_{f_p} := \underset{\theta \in [-\pi,\pi]}{\max} f_p(\theta) $.
\begin{theorem} \label{th:poly42}
Let $ p \geq 2 $. Then, the sequence of normalized collocation matrices $ \{ \frac{1}{n^2} A_{n,p}^{\spQ_{n\param}} \} $ is distributed like the function $ \kappa \otimes f_p^{\spQ_\param} $, where $f_p^{\spQ_\param}$ is defined in \eqref{c_fpH}-\eqref{c_fpT}, in the sense of the eigenvalues, i.e., for all $ F \in C_c({\mathbb{C}}) $,
$$   \lim_{n \rightarrow \infty} \frac{1}{n+p-2} 
     \sum_{j=1}^{n+p-2} F\left(\lambda_j\left(\frac{1}{n^2} A_{n,p}^{\spQ_{n\param}}\right)\right) 
   = \frac{1}{2\pi} \int_0^1 \int_{-\pi}^{\pi} F(\kappa(x) f_p^{\spQ_\param}(\theta)) 
     {\rm d}\theta {\rm d}x. $$
\end{theorem}
\begin{proof}
We decompose the expression of $ \frac{1}{n^2} A_{n,p}^{\spQ_{n\param}} $ as follows:
$$ \frac{1}{n^2} A_{n,p}^{\spQ_{n\param}} = \widetilde{D}_{n+p-2}(\kappa) \circ T_{n+p-2}(f_p^{\spQ_\param})
     + E_{n,p}^{\spQ_{n\param}} + F_{n,p}^{\spQ_{n\param}} + L_{n,p}^{\spQ_{n\param}}, $$
where:
\begin{align*}
E_{n,p}^{\spQ_{n\param}} & := \frac{1}{n} D_{n,p}^{\spQ_{n\param}}(\beta) H_{n,p}^{\spQ_{n\param}} +
                              \frac{1}{n^2} D_{n,p}^{\spQ_{n\param}}(\gamma) M_{n,p}^{\spQ_{n\param}}, \\
F_{n,p}^{\spQ_{n\param}} & := D_{n,p}^{\spQ_{n\param}}(\kappa) K_{n,p}^{\spQ_{n\param}} -
                              D_{n,p}^{\spQ_{n\param}}(\kappa) T_{n+p-2}(f_p^{\spQ_\param}),
\end{align*}
and:
$$ L_{n,p}^{\spQ_{n\param}} := D_{n,p}^{\spQ_{n\param}}(\kappa) T_{n+p-2}(f_p^{\spQ_\param}) -
                               \widetilde{D}_{n+p-2}(\kappa) \circ T_{n+p-2}(f_p^{\spQ_\param}). $$
The hypotheses of Theorem \ref{th:poly01} are satisfied with $ Z_n = \frac{1}{n^2} A_{n,p}^{\spQ_{n\param}} $, $ X_n = \widetilde{D}_{n+p-2}(\kappa) \circ T_{n+p-2}(f_p^{\spQ_\param}) $ and $ Y_n = E_{n,p}^{\spQ_{n\param}} + F_{n,p}^{\spQ_{n\param}} + L_{n,p}^{\spQ_{n\param}} $. \\
Indeed, since $ \kappa, f_p^{\spQ_\param} $ are real, and $ f_p^{\spQ_\param} $ is a trigonometric polynomial, from \cite[Corollary 2.6]{our-MATHCOMP} it follows that
$$ \{ \widetilde{D}_{n+p-2}(\kappa) \circ T_{n+p-2}(f_p^{\spQ_\param}) \} \sim_{\lambda}
   \kappa \otimes f_p^{\spQ_\param}. $$
We have also $ \Vert T_{n+p-2}(f_p^{\spQ_\param}) \Vert \leq M_{f_p^{\spQ_\param}} $, and moreover the band structure of $ T_{n+p-2}(f_p^{\spQ_\param}) $ implies that $ \Vert T_{n+p-2}(f_p^{\spQ_\param}) \Vert_{\infty} \leq C_T^{\spQ_\param} $. As a consequence, having both the matrices $ \widetilde{D}_{n+p-2}(\kappa) $ and $ T_{n+p-2}(f_p^{\spQ_\param}) $ Hermitian, their Hadamard product it is as well, and
$$       \Vert \widetilde{D}_{n+p-2}(\kappa) \circ T_{n+p-2}(f_p^{\spQ_\param}) \Vert
    \leq \Vert \widetilde{D}_{n+p-2}(\kappa) \circ T_{n+p-2}(f_p^{\spQ_\param}) \Vert_{\infty}
    \leq M_{\kappa} C_T^{\spQ_\param}, $$
where $ M_{\kappa} := \max_{x \in [0,1]} \kappa(x) $ and $ C_T^{\spQ_\param} $ are constants independent of $n$. \\
From Lemma \ref{lem:poly41} we have
$$ \Vert E_{n,p}^{\spQ_{n\param}} \Vert \leq M_{\vert \beta \vert} \frac{C_{p,\param}^{(1)}}{n}
   + M_{\gamma} \frac{1}{n^2} \sqrt{\frac{3p}{2}}, $$
and 
$$ \Vert E_{n,p}^{\spQ_{n\param}} \Vert_1 \leq (n+p-2) \Vert E_{n,p}^{\spQ_{n\param}} \Vert
   \leq (n+p-2) \left( M_{\vert \beta \vert} \frac{C_{p,\param}^{(1)}}{n} 
                     + M_{\gamma} \frac{1}{n^2} \sqrt{\frac{3p}{2}} \right), $$
where $ M_{\vert \beta \vert} := \max_{x \in [0,1]} \vert \beta(x) \vert $ and $ M_{\gamma} := \max_{x \in [0,1]} \gamma(x) $, implying that $ \Vert E_{n,p}^{\spQ_{n\param}} \Vert $ and $ \Vert E_{n,p}^{\spQ_{n\param}} \Vert_1 $ are bounded above by a constant independent of $n$. \\
The same lemma can be used to bound
$$ \Vert F_{n,p}^{\spQ_{n\param}} \Vert 
   \leq M_{\kappa} (\Vert K_{n,p}^{\spQ_{n\param}} \Vert + \Vert T_{n+p-2}(f_p^{\spQ_\param}) \Vert)
   \leq M_{\kappa} (C_{p,\param}^{(2)} + M_{f_p^{\spQ_\param}}), $$
and, since $ R_{n,p}^{\spQ_{n\param}} = K_{n,p}^{\spQ_{n\param}} - T_{n+p-2}(f_p^{\spQ_\param}) $, with $ \rank(R_{n,p}^{\spQ_{n\param}}) \leq 2 \lfloor 3p/2 \rfloor - 2 $, we obtain
\begin{align*}
  \Vert F_{n,p}^{\spQ_{n\param}} \Vert_1 & =
  \Vert D_{n,p}^{\spQ_{n\param}}(\kappa) R_{n,p}^{\spQ_{n\param}} \Vert_1
  \leq (2 \lfloor 3p/2 \rfloor - 2) \Vert D_{n,p}^{\spQ_{n\param}}(\kappa) R_{n,p}^{\spQ_{n\param}} \Vert \\
& \leq (2 \lfloor 3p/2 \rfloor - 2) M_{\kappa} (C_{p,\param}^{(2)} + M_{f_p}^{\spQ_\param} )
\end{align*}
thus having $ \Vert F_{n,p}^{\spQ_{n\param}} \Vert $ and $ \Vert F_{n,p}^{\spQ_{n\param}} \Vert_1 $ bounded above by a constant independent of $n$. \\
Finally, $ L_{n,p}^{\spQ_{n\param}} $ shares the same band structure as $ T_{n+p-2}(f_p^{\spQ_\param}) $, with a bandwidth of at most $p+1$, allowing us to express any nonzero element of it, at row $i$ and column $j$, by
$$ (L_{n,p}^{\spQ_{n\param}})_{i,j} = (T_{n+p-2}(f_p^{\spQ_\param}))_{i,j}
   \left( \kappa(\xi_{i+1,p}) - \kappa \left( \frac{\min(i,j)-1}{n+p-2} \right) \right), $$
where $ 0 \leq \vert i-j \vert \leq \lfloor \frac{p}{2} \rfloor $. Since
\begin{equation}
\left\vert \xi_{i+1,p} - \frac{\min(i,j)-1}{n+p-2} \right\vert \leq \frac{C}{n},
\end{equation}
we have
$$      \vert (L_{n,p}^{\spQ_{n\param}})_{i,j} \vert 
   \leq \vert (T_{n+p-2}(f_p^{\spQ_\param}))_{i,j} \vert \omega \left( \kappa,\frac{C}{n} \right). $$
It follows that both $ \Vert L_{n,p}^{\spQ_{n\param}} \Vert_{\infty} $ and $ \Vert( L_{n,p}^{\spQ_{n\param}})^T \Vert_{\infty} $ are bounded above by the expression $ \omega (\kappa,\frac{C}{n}) \Vert T_{n+p-2}(f_p^{\spQ_\param}) \Vert_{\infty} $, and so by \eqref{c_eq:norm2bound} we have
$$ \Vert L_{n,p}^{\spQ_{n\param}} \Vert \leq \omega \left( \kappa,\frac{C}{n} \right) \Vert T_{n+p-2}(f_p^{\spQ_\param}) \Vert_{\infty} \leq C_T^{\spQ_\param} \omega \left( \kappa,\frac{C}{n} \right). $$
This means that
\begin{equation} \label{c_eq:poly429}
     \frac{\Vert L_{n,p}^{\spQ_{n\param}} \Vert_1}{n+p-2} \leq \Vert L_{n,p}^{\spQ_{n\param}} \Vert
\leq C_T^{\spQ_\param} \omega \left( \kappa,\frac{C}{n} \right).
\end{equation}
Since $\kappa$ is continuous over the interval $[0,1]$, it holds that $ \lim_{n \rightarrow \infty} \omega (\kappa,\frac{C}{n}) = 0 $, and this implies
$$ \lim_{n \rightarrow \infty} \frac{\Vert L_{n,p}^{\spQ_{n\param}} \Vert_1}{n+p-2} = 0. $$
As a consequence of these facts, some constant $C_Y$ independent of $n$ exists such that both
\begin{align*}
\Vert Y_n \Vert = \Vert E_{n,p}^{\spQ_{n\param}} + F_{n,p}^{\spQ_{n\param}} + L_{n,p}^{\spQ_{n\param}} \Vert
             \leq \Vert E_{n,p}^{\spQ_{n\param}} \Vert + \Vert F_{n,p}^{\spQ_{n\param}} \Vert
                + \Vert L_{n,p}^{\spQ_{n\param}} \Vert & \leq C_Y, \\
\lim_{n \rightarrow \infty} 
\frac{\Vert E_{n,p}^{\spQ_{n\param}} + F_{n,p}^{\spQ_{n\param}} + L_{n,p}^{\spQ_{n\param}} \Vert_1}{n+p-2} & = 0.
\end{align*}
holds. So, all the hypotheses of Theorem \ref{th:poly01} are satisfied, allowing us to conclude that $ \{ \frac{1}{n^2} A_{n,p}^{\spQ_{n\param}} \} \sim_{\lambda} \kappa \otimes f_p^{\spQ_\param} $.
\end{proof}


\begin{remark} \label{rem:poly44}
Referring to \cite[Theorem 12]{GMPSS14}, \cite[Remark 3.2]{our-MATHCOMP}, and \cite[Remark 1]{galerkin_gbsp}, in the case of constant coefficients, while the phase parameter $\param$ tends to zero, the symbol derived for the normalized stiffness matrices in the collocation formulation with GB-splines of odd degree $2q+1$ approaches the one as for the normalized matrices in the Galerkin formulation with GB-splines of degree $q$. On the other side, equality does not hold for fixed $\param$.
\end{remark}

The next Lemma and the next Conjecture (partly) generalize \cite[Eqs. (3.19) and (3.25)]{our-MATHCOMP}, which state that both the value $h_p(\pi)$ and the ratio $ f_p(\pi) / M_{f_p} $ converge to zero exponentially as $ p \rightarrow \infty $, and specifically

\begin{align}
h_p(\pi) \leq \frac{h_p(\pi)}{h_p(\frac{\pi}{2})} & \leq 2^{\frac{1-p}{2}}, \label{c_eq:polycolloc319} \\
\frac{f_p(\pi)}{M_{f_p}} & \leq \frac{f_p(\pi)}{M_{f_p}} \leq 2^{\frac{5-p}{2}}. \label{c_eq:polycolloc325}
\end{align}

\begin{lemma} \label{lem:expdecay}
If $p$ is odd, for every value of $\param$ while $ \spQ = \spH $, for $ \param < \frac{\pi}{2} $ while $ \spQ = \spT $, we have that 
the ratio $ \frac{f_p^{\spQ_\param}(\pi)}{M_{f_p^{\spQ_\param}}} $ decays exponentially to zero as $ p \rightarrow \infty $. 
\end{lemma}

\begin{proof}
We first consider the case $ \spQ = \spH $. For every $ \param \in \RR^+, k \in \ZZ $, the following inequalities hold
\begin{align}
       \frac{\param^2}{\cosh(\param)-1} \frac{\cosh(\param)+1}{(\pi+2k\pi)^2 + \param^2}
& \leq \frac{(\param^2/4)(\cosh(\param)+1)}{\cosh(\param)-1} \frac{1}{(\pi/2+k\pi)^2}, \label{c_eq:expdecayh1} \\
       \frac{\param^2}{\cosh(\param)-1} \frac{\cosh(\param)}{(\pi/2+2k\pi)^2 + \param^2}
& \geq \frac{1/2}{(\pi/4+k\pi)^2}. \label{c_eq:expdecayh2}
\end{align}
This implies, due to \eqref{c_hpP}-\eqref{c_hpH}, \eqref{c_fpP}-\eqref{c_fpH}, and \eqref{c_eq:polycolloc319}-\eqref{c_eq:polycolloc325}, that for feasible $p$ a constant $ C_{\alpha} $, independent of $p$, exists such that
\begin{align}
h_p^{\spH_\param}(\pi) & \leq C_{\alpha} h_p(\pi) \leq C_{\alpha} 2^{\frac{1-p}{2}}, \label{c_eq:expdecayh3} \\
f_p^{\spH_\param}(\pi) & = 4 h_{p-2}^{\spH_\param}(\pi) \leq 4 C_{\alpha} h_{p-2}(\pi)
                           \leq C_{\alpha} 2^{\frac{7-p}{2}}, \label{c_eq:expdecayh4} \\
\frac{f_p^{\spH_\param}(\pi)}{M_{f_p^{\spH_\param}}} & \leq
\frac{f_p^{\spH_\param}(\pi)}{f_p^{\spH_\param}(\frac{\pi}{2})} \leq
\frac{C_{\alpha} f_p(\pi)}{f_p^{\spH_\param}(\frac{\pi}{2})} \leq
C_{\alpha} \frac{f_p(\pi)}{f_p(\frac{\pi}{2})} \leq C_{\alpha} 2^{\frac{5-p}{2}}, \label{c_eq:expdecayh5}
\end{align}
where \eqref{c_eq:expdecayh3} and \eqref{c_eq:expdecayh4} are a consequence of \eqref{c_eq:expdecayh1}, and \eqref{c_eq:expdecayh5} follows from \eqref{c_eq:expdecayh2}.
Now, we consider the case $ \spQ = \spT $. For every $ \param < \frac{\pi}{2}, k \in \ZZ $, the following inequalities hold
\begin{align}
       \frac{\param^2}{1-\cos(\param)} \frac{\cos(\param)+1}{(\pi+2k\pi)^2 - \param^2}
& \leq \frac{1}{(\pi/2+k\pi)^2}, \label{c_eq:expdecayt1} \\
       \frac{\param^2}{1-\cos(\param)} \frac{\cos(\param)}{(\pi/2+2k\pi)^2 - \param^2}
& \geq \frac{1/2}{(\pi/4+k\pi)^2} \frac{(\param^2/2) \cos(\param)}{1-\cos(\param)}. \label{c_eq:expdecayt2}
\end{align}
This implies, due to \eqref{c_hpP}, \eqref{c_hpT}-\eqref{c_fpP}, \eqref{c_fpT}, and \eqref{c_eq:polycolloc319}-\eqref{c_eq:polycolloc325}, that for feasible $p$ a constant $ K_{\alpha} $, independent of $p$, exists such that
\begin{align}
h_p^{\spT_\param}(\pi) & \leq h_p(\pi) \leq 2^{\frac{1-p}{2}}, \label{c_eq:expdecayt3} \\
f_p^{\spT_\param}(\pi) & = 4 h_{p-2}^{\spT_\param}(\pi) \leq 4 h_{p-2}(\pi) \leq 2^{\frac{7-p}{2}}, 
                           \label{c_eq:expdecayt4} \\
\frac{f_p^{\spT_\param}(\pi)}{M_{f_p^{\spT_\param}}} & \leq 
\frac{f_p^{\spT_\param}(\pi)}{f_p^{\spT_\param}(\frac{\pi}{2})} \leq 
\frac{f_p(\pi)}{f_p^{\spT_\param}(\frac{\pi}{2})} \leq
\frac{f_p(\pi)}{K_{\alpha} f_p(\frac{\pi}{2})} \leq \frac{1}{K_{\alpha}} 2^{\frac{5-p}{2}}. \label{c_eq:expdecayt5}
\end{align}
where \eqref{c_eq:expdecayt3} and \eqref{c_eq:expdecayt4} are a consequence of \eqref{c_eq:expdecayt1}, and \eqref{c_eq:expdecayt5} follows from \eqref{c_eq:expdecayt2}.
%
%
\end{proof}

\begin{conjecture} \label{conj:expdecay}
The ratio $ \frac{f_p^{\spQ_\param}(\pi)}{M_{f_p^{\spQ_\param}}} $ decays exponentially to zero as $ p \rightarrow \infty $ for every strictly increasing sequence of degrees, for both $ \spQ = \spH, \spT $, and for every feasible $\param$. Note that, since $ M_{f_p^{\spQ_\param}} \leq 4 $ for $ p \geq 4 $, also $ f_p^{\spQ_\param}(\pi) $ decays at least at the same rate, and so does $ h_p^{\spQ_\param}(\pi) = f_{p+2}^{\spQ_\param}(\pi) / 4 $. \\
This is strongly supported by a large number of numerical tests, which show the behavior described in Lemma \ref{lem:expdecay}, and proved there for certain sets of parameters, to hold in a more general setting.
\end{conjecture}

\begin{remark} \label{rem:poly45}
We see from Lemma \ref{lem:expdecay} and Conjecture \ref{conj:expdecay} that the symbol $ \kappa(x) f_p^{\spQ_\param} $ of the sequence of normalized collocation matrices has a zero at $ \theta = 0 $ and, moreover, a ``numerical zero'' at $ \theta = \pi $ for large $p$ and for any fixed $ x \in [0,1] $. This is implied by the fact that the convergence of $ f_p^{\spQ_\param}(\pi) $ to zero is monotonic and exponential with respect to $p$, so $ f_p^{\spQ_\param}(\pi) $ is very small for large $p$ and behaves numerically like a zero.
\end{remark}

\begin{remark} \label{rem:poly46}
As for the polynomial case studied in \cite{our-MATHCOMP}, the ``numerical zero'' of $ f_p^{\spQ_\param} $ at $ \theta = \pi $, which is present for larger values of $p$, negatively affects the convergence rate of standard multigrid methods. As a consequence, they behave unsatisfactily, also for values of $p$ which are quite moderate and not too large for practical use. In \cite[Table 4]{DGMSS14a}, we see that in the Galerkin formulation with B-splines of degree $q$, the convergence rate of standard multigrid methods, using Richardson smoothing, is already very slow for $ q \geq 6 $. By taking into account Remark \ref{rem:poly44}, in the collocation formulation with GB-splines of degree $ p = 2q+1 $, we may predict that the convergence rate of standard multigrid methods, at least for little values of the phase parameter $\alpha$, will be slow for $ p \geq 13 $. 
\end{remark}

\begin{remark} \label{rem:poly47}
When the advection coefficient $\beta$ is large with respect to both the diffusion coefficient $\kappa$ and the fineness parameter $n$, the spectrum of the collocation matrix  $ \frac{1}{n^2} A_{n,p}^{\spQ_{n\param}} $ may be badly approximated by its theoretical asymptotic distribution. However, for sufficiently large $n$ and fixed $\beta$, the effects of the advection disappear and the theoretical asymptotic distribution is approached (because $ \frac{1}{n} D_{n,p}^{\spQ_{n\param}}(\beta) H_{n,p}^{\spQ_{n\param}} \approx O $ and $ \frac{1}{n^2} A_{n,p}^{\spQ_{n\param}} \approx D_{n,p}^{\spQ_{n\param}}(\kappa) K_{n,p}^{\spQ_{n\param}} $; see the proof of Theorem \ref{th:poly42}). \\
Conversely, if we assume that $\beta$ grows proportionally with $n$ (this is theoretically unfeasible, but can occur in certain practical situations), then the symbol structurally changes and becomes complex-valued. This has important implications on the spectrum and requires further investigation.
\end{remark}

\subsubsection{Use of a geometry map within this setting} \label{geo_nonnested}

The general case of isogeometric collocation methods with a non-trivial geometry map can be easily addressed with the aid of the results from the previous subsection. Indeed, the geometry map only invokes a change of variable, and leads to a new formulation of the problem with different coefficients. More precisely, given a geometry map $ G : [0,1] \rightarrow [0,1] $, saying that $u$ satisfies our model problem \eqref{c_eq:polycolloc401} is equivalent to say that the corresponding function $ \hat{u} := u(G) $, defined on the parametric domain $[0,1]$, satisfies the transformed problem
$$ \left\{ \begin{aligned}
& - \frac{\kappa(G(\hat{x}))}{(G'(\hat{x}))^2} \hat{u}''(\hat{x})
  + \left( \frac{\kappa(G(\hat{x})) G''(\hat{x})}{(G'(\hat{x}))^3}
          +\frac{\beta(G(\hat{x}))}{G'(\hat{x})} \right) \hat{u}'(\hat{x})
  + \gamma(G(\hat{x})) \hat{u}(\hat{x}) = \hbox{f}(G(\hat{x})), \\
& \hat{u}(0)=0, \quad \hat{u}(1)=0,
\end{aligned} \right. $$
with $ 0 < \hat{x} < 1 $. In this case, our matrix $ A = A_{n,p}^{\spQ_{n\param}} $ is equal to
\begin{align*}
     A_{n,p}^{\spQ_{n\param}} 
& = \left[ - \frac{\kappa(G(\xi_{i+1,p}))}{(G'(\xi_{i+1,p}))^2} 
             (N_{j+1,p}^{\spQ_{n\param}})'' (\xi_{i+1,p}) \right. \\
& \qquad   + \left( \frac{\kappa(G(\xi_{i+1,p})) G''(\xi_{i+1,p})}{(G'(\xi_{i+1,p}))^3}
           + \frac{\beta(G(\xi_{i+1,p}))}{G'(\xi_{i+1,p})} \right)
             (N_{j+1,p}^{\spQ_{n\param}})' (\xi_{i+1,p}) \\
& \qquad   + \gamma(G(\xi_{i+1,p})) N_{j+1,p}^{\spQ_{n\param}} (\xi_{i+1,p}) 
             \left. \vphantom{\frac{1}{1}} \right]_{i,j=1}^{n+p-2},
\end{align*}
while the right-hand side is $ {\bf f} = [\hbox{f}(G(\xi_{i+1,p}))]_{i=1}^{n+p-2} $. \\
The next theorem shows explicitly the influence of the geometry map on the spectral distribution of the normalized matrices $ \{ \frac{1}{n^2} A_{n,p}^{\spQ_{n\param}} \} $.
\begin{theorem} \label{th:poly48}
Let $ p \geq 2 $. Let $ G : [0,1] \rightarrow [0,1] $ such that $ G \in C^1([0,1]) $, $ 0 < G'(\hat{x}) $ for all $ \hat{x} \in [0,1] $ and $G''$ is bounded. Then, the sequence of normalized collocation matrices $ \{ \frac{1}{n^2} A_{n,p}^{\spQ_{n\param}} \} $ is distributed like the function $ \frac{\kappa(G)}{(G')^2} \otimes f_p^{\spQ_\param} $. 
\end{theorem}

\begin{proof}
This is a straightforward consequence of Theorem \ref{th:poly42}, by comparing the two forms of the matrices $ A_{n,p}^{\spQ_{n\param}} $.
\end{proof}

\begin{remark} \label{rem:poly49}
The geometry map $G$ which is considered in Theorem \ref{th:poly48} can be given in any representation: the GB-spline form is prescribed by the paradigm of isogeometric analysis, but the theorem holds for a larger set of choices for it.
\end{remark}

\subsection{The nested case}

\begin{lemma} \label{c_lem:L1convergence}
For $p\geq2$ and sufficiently large $n$, it holds
\begin{equation} \label{c_eq:L1convergence}
\Bigl\Vert f_{p}^{\spQ_\fracparam}-f_{p} \Bigr\Vert_{L_1([-\pi, \pi])}\leq C\,n^{-2}, \quad \spQ=\spH,\spT,
\end{equation}
where $C$ is a constant independent of $n$.
\end{lemma}
\begin{proof}
For $p=2,3,4$, it is a consequence of \eqref{c_expl:gf2T} and \eqref{c_expl:gf2H}-\eqref{c_expl:f4P}. For $ p \geq 5 $, we just prove the hyperbolic case, because the trigonometric case can be addressed with similar arguments. From \eqref{c_fpP}-\eqref{c_fpH} we obtain
{\small \begin{align*}
  &  \qquad \vert f_{p}^{\spH_\fracparam}(\theta)-f_{p}(\theta) \vert \\
  &\leq 4 \sum_{k \in \ZZ} \left| \frac{{(\sin(\theta/2+k\pi))}^{p-1}}{{(\theta/2+k\pi)}^{p-3}} \right|
     \left[ \left(\frac{(\fracparam)^2}{\cosh(\fracparam)-1}\right)
     \left(\frac{\cosh(\fracparam)-\cos(\theta)}{(\theta+2k\pi)^2+(\fracparam)^2}\right) 
   - \left( \frac{{\sin(\theta/2+k\pi)}}{{\theta/2+k\pi}} \right)^2 \right] \\ 
  & \leq 4 \sum_{k \in \ZZ} \left(\frac{(\fracparam)^2}{\cosh(\fracparam)-1}\right)
     \left(\frac{\cosh(\fracparam)-\cos(\theta)}{(\theta+2k\pi)^2+(\fracparam)^2}\right) 
   - \left( \frac{{\sin(\theta/2+k\pi)}}{{\theta/2+k\pi}} \right)^2.
\end{align*} } 
We can write:
$$ \left(\frac{(\fracparam)^2}{\cosh(\fracparam)-1}\right)
   \left(\frac{\cosh(\fracparam)-\cos(\theta)}{(\theta+2k\pi)^2+(\fracparam)^2}\right) 
 - \left( \frac{{\sin(\theta/2+k\pi)}}{{\theta/2+k\pi}} \right)^2
:= l^{\spH}(\fracparam,\theta+2k\pi), $$
by setting:
$$ l^{\spH}(\varepsilon,\eta)
:= \left(\frac{\varepsilon^2}{\cosh(\varepsilon)-1}\right)
   \left(\frac{\cosh(\varepsilon)-\cos(\eta)}{\eta^2+\varepsilon^2}\right) 
 - \left( \frac{\sin(\eta/2)}{\eta/2} \right)^2. $$
We consider the Taylor expansion of $l^{\spH}$ with respect to $\varepsilon$ up to the second degree at the point $0$, i.e.,
\begin{equation*}
 \varepsilon^2
\frac{\eta^2 \cos(\eta) + 5\eta^2 + 12 \cos(\eta) - 12}{6 \eta^4}
=:\varepsilon^2 \tilde{l}^{\spH}(\eta).
\end{equation*}
It can be verified that $|\tilde{l}^{\spH}(\eta)|\leq \min(0.04,\eta^{-2})$.
As a consequence, we have for sufficiently small $\varepsilon$,
$$
|l^{\spH}(\varepsilon,\eta)| \leq \varepsilon^2 \widetilde C \min(0.04,\eta^{-2}),
$$
for some positive constant $\widetilde C$ independent of $\varepsilon$ and $\eta$. \\
Using the previous bounds, we obtain for sufficiently large $n$,
\begin{align*}
\vert f_{p}^{\spH_\fracparam}(\theta)-f_{p}(\theta) \vert
&\leq
4 \left( \sum_{|k| \leq 1} \Bigl| l^{\spH}\Bigl(\frac{\param}{n},\theta+2k\pi \Bigr) \Bigr| 
  + \sum_{|k| \geq 2} \Bigl| l^{\spH}\left(\frac{\param}{n},\theta+2k\pi \right) \Bigr| \right) \\
& \leq 4\left(\frac{\param}{n}\right)^2 \widetilde C 
  \Bigl(0.12 + \sum_{|k| \geq 2} \frac{1}{(\theta+2k\pi)^2} \Bigr) \\
& \leq 4\left(\frac{\param}{n}\right)^2 \widetilde C 
  \Bigl(0.12 + \frac{2}{\pi^2}\sum_{k \geq 2} \frac{1}{(2k-1)^2} \Bigr) \leq \overset{\vee}{C} n^{-2},
\end{align*}
which implies (\ref{c_eq:L1convergence}).
\end{proof}
\begin{lemma} \label{lem:poly41_nested}
For every $ p \geq 2 $ and every $ n \geq 2 $, we have
$$ \Vert M_{n,p}^{\spQ_\param} \Vert_2 \leq \sqrt{\frac{3p}{2}}, \quad
   \Vert H_{n,p}^{\spQ_\param} \Vert_2 \leq C_{p,\param}^{(3)}, \quad
   \Vert K_{n,p}^{\spQ_\param} \Vert_2 \leq C_{p,\param}^{(4)}, $$
where $ C_{p,\param}^{(3)} $ and $ C_{p,\param}^{(4)} $ are constants independent of $n$.
\end{lemma}
\begin{proof}
By referring to Lemma \ref{lem:poly41}, the limitation about $ \Vert M_{n,p}^{\spQ_{\param}} \Vert_2 $ does not depend on $\param$, and so it is proved there. Limitations regarding $ \Vert H_{n,p}^{\spQ_{\param}} \Vert_2 $ and $ \Vert K_{n,p}^{\spQ_{\param}} \Vert_2 $ derive by setting, if $\spQ=\spH$
\begin{align*}
C_{p,\param}^{(3)} &:= \underset{\overline{\param} \in (0,\param]}{\sup} C_{p,\overline{\param}}^{(1)} 
                   \geq \sup_{n \in \NN_0} C_{p,\fracparam}^{(1)}, \\
C_{p,\param}^{(4)} &:= \underset{\overline{\param} \in (0,\param]}{\sup} C_{p,\overline{\param}}^{(2)}
                   \geq \sup_{n \in \NN_0} C_{p,\fracparam}^{(2)},
\end{align*}
in which $ C_{p,\param}^{(3)} $ and $ C_{p,\param}^{(4)} $ are finite because $ C_{p,\overline{\param}}^{(1)} $ and $ C_{p,\overline{\param}}^{(2)} $ are continuous as functions of $\overline{\param}$, and both $ \underset{n \rightarrow \infty}{\lim} C_{p,\fracparam}^{(1)} = \underset{\overline{\param} \rightarrow 0}{\lim} C_{p,\overline{\param}}^{(1)} $ and $ \underset{n \rightarrow \infty}{\lim} C_{p,\fracparam}^{(2)} = \underset{\overline{\param} \rightarrow 0}{\lim} C_{p,\overline{\param}}^{(2)} $ are finite, tending to polynomial case bounds. \\
If $\spQ=\spT$, we have to take care of the fact that, if $ \param \geq \pi $, not every value of $ n \in \NN_0 $ is feasible for construction, due to the constraint in \eqref{Trig}. More specifically, it is necessary to consider $ n \geq \lfloor \frac{\param}{\pi} \rfloor + 1 $; if we define $ \param^* = \frac{\param}{\lfloor \frac{\param}{\pi} \rfloor + 1} $, we can set
\begin{align*}
C_{p,\param}^{(3)} &:= \underset{\overline{\param} \in (0,\param^*]}{\sup} C_{p,\overline{\param}}^{(1)} 
                   \geq \sup_{n \geq \lfloor \frac{\param}{\pi} \rfloor + 1} C_{p,\fracparam}^{(1)}, \\
C_{p,\param}^{(4)} &:= \underset{\overline{\param} \in (0,\param^*]}{\sup} C_{p,\overline{\param}}^{(2)}
                   \geq \sup_{n \geq \lfloor \frac{\param}{\pi} \rfloor + 1} C_{p,\fracparam}^{(2)}.
\end{align*}
Nevertheless, considerations on finiteness of $ C_{p,\param}^{(3)} $ and $ C_{p,\param}^{(4)} $ are analogous.
\end{proof}
%
%
In the following Theorem, that is the main one for the nested case, we will use the following notation, in which, as in the proof of the previous Lemma, $ \param^* = \frac{\param}{\lfloor \frac{\param}{\pi} \rfloor + 1} $.
\begin{align*}
M_{f_p^{\spH_{[0,\param]}}} &:= \sup_{\overline{\param} \in (0,\param]} M_{f_p^{\spH_{\overline{\param}}}}
                           \geq \sup_{n \in \NN_0} M_{f_p^{\spH_\fracparam}}, \\
C_T^{\spH_{[0,\param]}}     &:= \sup_{\overline{\param} \in (0,\param]} C_T^{\spH_{\overline{\param}}}
                           \geq \sup_{n \in \NN_0} C_T^{\spH_\fracparam}, \\
M_{f_p^{\spT_{[0,\param]}}} &:= \sup_{\overline{\param} \in (0,\param^*]} M_{f_p^{\spT_{\overline{\param}}}}
                           \geq \sup_{n \geq \lfloor \frac{\param}{\pi} \rfloor + 1} M_{f_p^{\spT_\fracparam}}, \\
C_T^{\spT_{[0,\param]}}     &:= \sup_{\overline{\param} \in (0,\param^*]} C_T^{\spT_{\overline{\param}}}
                           \geq \sup_{n \geq \lfloor \frac{\param}{\pi} \rfloor + 1} C_T^{\spT_\fracparam}.
\end{align*}
Note that $ M_{f_p^{\spQ_{[0,\param]}}} $ and $ C_T^{\spQ_{[0,\param]}} $ are finite because $ M_{f_p^{\spQ_{\overline{\param}}}} $ and $ C_T^{\spQ_{\overline{\param}}} $ are continuous as functions of $\overline{\param}$, $ \underset{n \rightarrow \infty}{\lim} M_{f_p^{\spQ_\fracparam}} = \underset{\overline{\param} \rightarrow 0}{\lim} M_{f_p^{\spQ_{\overline{\param}}}} = M_{f_p} $, and $ \underset{n \rightarrow \infty}{\lim} C_T^{\spQ_\fracparam} = \underset{\overline{\param} \rightarrow 0}{\lim} C_T^{\spQ_{\overline{\param}}} = C_T $.

\begin{theorem} \label{th:poly42_nested}
Let $ p \geq 2 $. Then, the sequence of normalized collocation matrices $ \{ \frac{1}{n^2} A_{n,p}^{\spQ_\param} \} $ is distributed like the function $ \kappa \otimes f_p $, where $f_p$ is defined in \eqref{c_fpP}, in the sense of the eigenvalues, i.e., for all $ F \in C_c({\mathbb{C}}) $,
$$   \lim_{n \rightarrow \infty} \frac{1}{n+p-2} 
     \sum_{j=1}^{n+p-2} F\left(\lambda_j\left(\frac{1}{n^2} A_{n,p}^{\spQ_\param}\right)\right) 
   = \frac{1}{2\pi} \int_0^1 \int_{-\pi}^{\pi} F(\kappa(x) f_p(\theta)) {\rm d}\theta {\rm d}x. $$
\end{theorem}
\begin{proof}
Let us recall
\begin{align*}
E_{n,p}^{\spQ_{\param}} & := \frac{1}{n} D_{n,p}^{\spQ_{\param}}(\beta) H_{n,p}^{\spQ_{\param}} +
                              \frac{1}{n^2} D_{n,p}^{\spQ_{\param}}(\gamma) M_{n,p}^{\spQ_{\param}}, \\
F_{n,p}^{\spQ_{\param}} & := D_{n,p}^{\spQ_{\param}}(\kappa) K_{n,p}^{\spQ_{\param}} -
                              D_{n,p}^{\spQ_{\param}}(\kappa) T_{n+p-2}(f_p^{\spQ_\fracparam}), \\
L_{n,p}^{\spQ_{\param}} & := D_{n,p}^{\spQ_{\param}}(\kappa) T_{n+p-2}(f_p^{\spQ_\fracparam}) -
                               \widetilde{D}_{n+p-2}(\kappa) \circ T_{n+p-2}(f_p^{\spQ_\fracparam}).
\end{align*}

For arguments used in the proofs of Theorem \ref{th:poly42} and Lemma \ref{lem:poly41_nested},
the following bounds hold
\begin{align}
& \Vert E_{n,p}^{\spQ_{\param}} \Vert \leq M_{\vert \beta \vert} \frac{C_{p,\param}^{(3)}}{n}
   + M_{\gamma} \frac{1}{n^2} \sqrt{\frac{3p}{2}}, \notag \\
& \Vert E_{n,p}^{\spQ_{\param}} \Vert_1 \leq (n+p-2) \Vert E_{n,p}^{\spQ_{\param}} \Vert
   \leq (n+p-2) \left( M_{\vert \beta \vert} \frac{C_{p,\param}^{(3)}}{n} 
                     + M_{\gamma} \frac{1}{n^2} \sqrt{\frac{3p}{2}} \right), \notag \\
& \Vert F_{n,p}^{\spQ_{\param}} \Vert 
   \leq M_{\kappa} (\Vert K_{n,p}^{\spQ_{\param}} \Vert + \Vert T_{n+p-2}(f_p^{\spQ_\fracparam}) \Vert)
   \leq M_{\kappa} (C_{p,\param}^{(4)} + M_{f_p^{\spQ_{[0,\param]}}}), \notag \\
& \Vert F_{n,p}^{\spQ_{\param}} \Vert_1 
   \leq (2 \lfloor 3p/2 \rfloor - 2) \Vert D_{n,p}^{\spQ_{\param}}(\kappa) R_{n,p}^{\spQ_{\param}} \Vert
   \leq (2 \lfloor 3p/2 \rfloor - 2) M_{\kappa} (C_{p,\param}^{(4)} + M_{f_p^{\spQ_{[0,\param]}}} ),
   \notag \\
& \Vert L_{n,p}^{\spQ_{\param}} \Vert 
   \leq \omega \left( \kappa,\frac{C}{n} \right) \Vert T_{n+p-2}(f_p^{\spQ_\fracparam}) \Vert_{\infty} 
   \leq C_T^{\spQ_{[0,\param]}} \omega \left( \kappa,\frac{C}{n} \right), \notag
\end{align}
and, along with the limit 
$$ \lim_{n \rightarrow \infty} \frac{\Vert L_{n,p}^{\spQ_{\param}} \Vert_1}{n+p-2} = 0, $$
it follows that a constant $K_Y$ independent of $n$ exists such that both
\begin{align}
     \Vert E_{n,p}^{\spQ_{\param}} + F_{n,p}^{\spQ_{\param}} + L_{n,p}^{\spQ_{\param}} \Vert
\leq \Vert E_{n,p}^{\spQ_{\param}} \Vert + \Vert F_{n,p}^{\spQ_{\param}} \Vert
   + \Vert L_{n,p}^{\spQ_{\param}} \Vert & \leq K_Y, \label{eflinf_nested} \\
\lim_{n \rightarrow \infty} 
\frac{\Vert E_{n,p}^{\spQ_{\param}} + F_{n,p}^{\spQ_{\param}} + L_{n,p}^{\spQ_{\param}} \Vert_1}{n+p-2} 
& = 0. \label{efl1_nested}
\end{align}
We say now that a sequence of matrices $G_n$ is distributed like a sequence of matrices $W_n$, $ G_n \sim_{\lambda} W_n $, if and only if a function $\varphi$ exists such that both $G_n$ and $W_n$ are distributed like $\varphi$. If we set
\begin{align*}
Z_n & = \tilde{D}_{n+p-2}(\kappa) \circ T_{n+p-2}(f_p^{\spQ_\fracparam}), \\
X_n & = \tilde{D}_{n+p-2}(\kappa) \circ T_{n+p-2}(f_p), \\
Y_n & = \tilde{D}_{n+p-2}(\kappa) \circ T_{n+p-2}(f_p^{\spQ_\fracparam} - f_p),
\end{align*}
then we obtain the decomposition
\begin{equation} \label{dlike_prechain}
\frac{1}{n^2} A_{n,p}^{\spQ_\param} = Z_n + E_{n,p}^{\spQ_\param} + F_{n,p}^{\spQ_\param} + L_{n,p}^{\spQ_\param}. 
\end{equation}
We want to prove the chain
\begin{equation} \label{dlike_chain}
\frac{1}{n^2} A_{n,p}^{\spQ_\param} \sim_{\lambda} Z_n \sim_{\lambda} X_n 
\sim_{\lambda} \kappa \otimes f_p.
\end{equation}
Since we find (see above about $Z_n$, and proof of \cite[Theorem 4.2]{our-MATHCOMP} about $X_n$) 
\begin{align*}
\Vert Z_n \Vert & = \Vert \tilde{D}_{n+p-2}(\kappa) \circ T_{n+p-2}(f_p^{\spQ_\fracparam}) \Vert
               \leq \Vert \tilde{D}_{n+p-2}(\kappa) \circ T_{n+p-2}(f_p^{\spQ_\fracparam}) \Vert_{\infty}
               \leq M_{\kappa} C_T^{\spQ_{[0,\param]}}, \\ 
\Vert X_n \Vert & = \Vert \tilde{D}_{n+p-2}(\kappa) \circ T_{n+p-2}(f_p) \Vert 
               \leq \Vert \tilde{D}_{n+p-2}(\kappa) \circ T_{n+p-2}(f_p) \Vert_{\infty} \leq M_{\kappa} C_T, 
\\ 
\Vert Y_n \Vert_1 & = \Vert \tilde{D}_{n+p-2}(\kappa) \circ T_{n+p-2}(f_p^{\spQ_\fracparam}-f_p) \Vert_1 \\
              & \leq (n+p-2) \Vert \tilde{D}_{n+p-2}(\kappa) \circ T_{n+p-2}(f_p^{\spQ_\fracparam}-f_p) \Vert
               \\ & \leq (n+p-2) \Vert \tilde{D}_{n+p-2}(\kappa) \circ T_{n+p-2}(f_p^{\spQ_\fracparam}-f_p) \Vert_{\infty} \\
              & \leq (n+p-2) M_{\kappa} \Vert T_{n+p-2}(f_p^{\spQ_\fracparam}-f_p) \Vert_{\infty} \\
              & \leq (n+p-2) M_{\kappa} \Vert f_p^{\spQ_\fracparam}-f_p \Vert_{L_{\infty}([-\pi,\pi])}
                \leq M_{\kappa} C_{\infty} \frac{n+p-2}{n^2} \leq (p-1) M_{\kappa} C_{\infty},
\end{align*}
we can verify every relation of the chain \eqref{dlike_chain} as follows: $ \frac{1}{n^2} A_{n,p}^{\spQ_\param} \sim_{\lambda} Z_n $ descends from \eqref{eflinf_nested}-\eqref{dlike_prechain} and Theorem \ref{th:poly01}, $ Z_n \sim_{\lambda} X_n $ is a consequence of \eqref{c_eq:L1convergence}, and finally $ X_n \sim_{\lambda} \kappa \otimes f_p $ is a known result from the polynomial case (see proof of \cite[Theorem 4.2]{our-MATHCOMP}).
\end{proof}



\begin{remark} \label{rem:poly44_nested}
In the case of constant coefficients, the symbol derived for the normalized stiffness matrices in the collocation formulation with GB-splines of odd degree $2q+1$ is the same as for the normalized matrices in the Galerkin formulation with GB-splines of degree $q$.
\end{remark}

\begin{remark} \label{rem:poly4x_nested} 
Remarks \ref{rem:poly45}-\ref{rem:poly47} also hold in the nested case, provided that $f_p^{\spQ_\param}$ is replaced by $f_p$, and $\paramgen=\param$ instead of $\paramgen=n\param$ in the matrices appearing in the statement of Remark \ref{rem:poly47}.
\end{remark}

\subsubsection{Use of a geometry map within this setting} \label{geo_nested}
\begin{theorem} \label{th:poly48_nested}
Let $ p \geq 2 $. Let $ G : [0,1] \rightarrow [0,1] $ such that $ G \in C^1([0,1]) $, $ 0 < G'(\hat{x}) $ for all $ \hat{x} \in [0,1] $ and $G''$ is bounded. Then, the sequence of normalized collocation matrices $ \{ \frac{1}{n^2} A_{n,p}^{\spQ_\param} \} $ is distributed like the function $ \frac{\kappa(G)}{(G')^2} \otimes f_p $. 
\end{theorem}
\begin{proof}
It reads in the same way as Theorem \ref{th:poly48} because of the influence of $G$ only on $\kappa$ and not on $f_p$, thus avoiding to interfere with the quadratic convergence of $ f_{p}^{\spQ_\fracparam} $ to $f_p$.
\end{proof}
\begin{remark} \label{rem:poly49_nested}
The geometry map $G$ can be given in any representation: the GB-spline form is prescribed by the paradigm of isogeometric analysis, but the theorem holds for a larger set of choices for it.
\end{remark}

\section{The multivariate setting} \label{sec:colmD}

The analysis carried out for the spectral results in Section \ref{sec:col1D} can be extended to the multi-dimensional setting, thanks to the tensor-product structure of the GB-splines. In this case, the linear elliptic differential problem \eqref{c_eq:polycolloc401} reads as:
\begin{equation} \left\{ \begin{aligned} \label{c_eq:polycolloc101}
& -\nabla \cdot K \nabla u + \bphi \cdot \nabla u + \gamma u = \hbox{f}, \qquad \hbox{in } \Omega, \\
& u = 0 \qquad \hbox{on } \partial\Omega,
\end{aligned} \right. 
\end{equation}
where $\Omega$ is a bounded open domain in $\RR^d$, $ K : \hbox{cl}(\Omega) \rightarrow \RR^{d \times d} $ is a symmetric positive definite (SPD) matrix of functions in $ C^1(\Omega) \cap C(\hbox{cl}(\Omega)) $, $ \bphi : \hbox{cl}(\Omega) \rightarrow \RR^d $ is a vector of functions in $ C(\hbox{cl}(\Omega)) $, $ \gamma, \hbox{f} \in C(\hbox{cl}(\Omega)) $ and $ \gamma \geq 0 $. \\
The problem \eqref{c_eq:polycolloc101} can be reformulated as follows:
\begin{equation} \left\{ \begin{aligned} \label{c_eq:polycolloc201}
& -\ve (K \circ Hu) \ve^T + \bbeta \cdot \nabla u + \gamma u = \hbox{f}, \qquad \hbox{in } \Omega, \\
& u = 0 \qquad \hbox{on } \partial\Omega,
\end{aligned} \right. 
\end{equation}
where $ \ve := [1 \cdots 1] $, $Hu$ denotes the Hessian of $u$, that is
$$ (Hu)_{i,j} := \frac{\partial^2 u}{\partial x_i \partial x_j}, $$
and $\circ$ denotes the componentwise Hadamard product (see \cite{bhatia97}); moreover, $\bbeta$ collects the coefficients of the first order derivatives in \eqref{c_eq:polycolloc101}, which are
$$ \beta_j := \alpha_j - \sum_{i=1}^d \frac{\partial \kappa_{i,j}}{\partial x_i}, $$
with $\kappa_{i,j}$ the entries of the matrix $ K := [\kappa_{i,j}]_{i,j=1}^d $. \\
The set of collocation points \eqref{c_eq:polycolloc208} also is to be considered multidimensional, and so are the Greville abscissae \eqref{c_eq:polycolloc406}. \\
The discretization aims at obtaining a system of the form $ A\bfu=\bff $, as in \eqref{collocation_system}. In general, the multivariate form of $A$ is given by
$$ [-\ve(K(\btau_i) \circ H \varphi_j (\btau_i)) \ve^T + \bbeta(\btau_i) \cdot \nabla \varphi_j(\btau_i)
   + \gamma(\btau_i) \varphi_j(\btau_i)]_{i,j=1}^N \in \RR^{N \times N}. $$
%
%
For a compact and clean description of the results, throughout this section, we use the multi-index notation, see \cite[Section~6]{Ty96} (or \cite[Section~2.1]{Full-Galerkin}).
A multi-index $\mm\in\ZZ^d$ is  a (row) vector in $\ZZ^d$ and its components are denoted by $m_1,\ldots,m_d$.
We indicate by $\mathbf0,\,\mathbf1,\,\mathbf2$ the vectors consisting of all zeros, all ones, all twos, respectively.
We set $N(\mm):=\prod_{i=1}^dm_i$ and $\jj\kk:=(j_1k_1,\ldots,j_dk_d)$. 
The inequality $\jj\le\kk$ means that $j_i\le k_i$ for $i=1,\ldots,d$ and the multi-index range $\jj,\ldots,\kk$ is the set $\{\iii\in\ZZ^d:\jj\le\iii\le\kk\}$. We assume for this set the standard lexicographic ordering:
\begin{equation}\label{ordering}
\left[\ \ldots\ \left[\ \left[\ (i_1,\ldots,i_d)\ \right]_{i_d=j_d,\ldots,k_d}\ \right]_{i_{d-1}=j_{d-1},\ldots,k_{d-1}}\ \ldots\ \right]_{i_1=j_1,\ldots,k_1}.
\end{equation}
We write $\iii=\jj,\ldots,\kk$ to denote that the
multi-index $\iii$ varies in the multi-index range $\jj,\ldots,\kk$ taking all the values from $\jj$ to $\kk$ following the ordering in \eqref{ordering}. 
For example, if $\mm\in\NN^d$ and $\bfx=[x_\iii]_{\iii=\mathbf1}^\mm$, then $\bfx$ is a vector of length $N(\mm)$ whose components $x_\iii,\ \iii=\mathbf1,\ldots,\mm$, are ordered according to \eqref{ordering}: the first component is $x_{\mathbf1}=x_{(1,\ldots,1,1)}$, the second component is $x_{(1,\ldots,1,2)}$, and so on. \\
%
%
Let also $\pp:=(p_1,\ldots,p_d)$ and $\nn:=(n_1,\ldots,n_d)$ be multi-indices such that $p_i,n_i\geq1$, $i=1,\ldots,d$. 
Let $\U:=(U_1,\ldots,U_d)$ and $\V:=(V_1,\ldots,V_d)$ be vectors of functions such that for $i=1,\ldots,d$ the pair of functions $ \{ U_i^{(p-1)}, V_i^{(p-1)} \} $  is a Tchebycheff system in $ [t_{i,j}, t_{i,j+1}] $, $j=p_i+1,\ldots,p_i+n_i$, where
$$ \{t_{i,1},\ldots,t_{i,n_i+2p_i+1}\}:=
\biggl\{\underbrace{0,\ldots,0}_{p_i+1},\frac{1}{n_i},\frac{2}{n_i},\ldots,
\frac{n_i-1}{n_i},\underbrace{1,\ldots,1}_{p_i+1}\biggr\}. $$
The tensor-product GB-splines $N_{\iii,\pp}^{\U,\V}:[0,1]^d\to\RR$ are defined by
\begin{equation} \label{tensor-product-B-splines}
N_{\iii,\pp}^{\U,\V} := N_{i_1,p_1}^{U_1,V_1} \otimes \cdots \otimes N_{i_d,p_d}^{U_d,V_d},
\quad \iii = \mathbf2, \ldots, \nn+\pp-\mathbf1. \\
\end{equation}
The use of the Greville abscissae as collocation points $\btau_i$, and of the tensor-product GB-splines as in \eqref{tensor-product-B-splines} as functions $\varphi_j$, $j=1,\ldots,N$ (having $N=N(\nn+\pp-\mathbf2)$), with the standard lexicographical ordering\footnote{A different ordering has been used in \cite{our-MATHCOMP}.} of \eqref{ordering}, allows us to obtain an expression which can be seen as the multidimensional version of \eqref{c_eq:polycolloc409}, in the case modified by a generic geometry map. \\
We can denote this matrix as $ A := A_{\nn,\pp}^{\U,\V} $, which becomes $ A_{\nn,\pp}^{\spQQ_\bmu} $ when hyperbolic or trigonometric GB-splines with parameters $\bmu:=(\mu_1, \ldots, \mu_d)$ are considered; note that different section spaces can be used in different directions. \\
In a fashion which is almost completely similar to the study of \cite[Section 6]{galerkin_gbsp}, we state that $\bparam:=(\param_1, \ldots, \param_d)$ is a suitable sequence of real, fixed parameters not depending on $n$, and that we assume $n_j=\nu_j n\in\NN$ for each $j=1,\ldots,d$, or $\nn=\bnu n$ for a fixed vector $\bnu:=(\nu_1,\ldots,\nu_d)$; note also that $\nn\bparam:=(n_1 \param_1, \ldots, n_d \param_d)$.
In addition, let it be $ {\bf G} : [0,1]^d \rightarrow \hbox{cl}(\Omega) $ the multidimensional geometry map (with $ {\bf G} \in C^1([0,1]^d) $, $K$ the coefficient matrix of our problem \eqref{c_eq:polycolloc201} (or equivalently \eqref{c_eq:polycolloc101}), and $J$ the Jacobian matrix of $ {\bf G} $. With all these ingredients, we can state the two key spectral theorems.
\begin{theorem} \label{c_distro_nested}
The sequence of normalized matrices $\{n^{-2} A_{\nn,\pp}^{\spQQ_\bparam}\}_n$ is distributed, in the sense of the eigenvalues, like the function
\begin{equation} \label{c_symbol_mD_nested}
\bnu (J^{-1} K({\bf G}) J^{-T} \circ H_{p_1,\ldots,p_d}) \bnu^T,
\end{equation}
where
$$ H_{p_1,\ldots,p_d} := \begin{bmatrix}
f_{p_1} \otimes h_{p_2} \otimes \cdots \otimes h_{p_d} & g_{p_1} \otimes g_{p_2} \otimes \cdots \otimes h_{p_d} &
\cdots & g_{p_1} \otimes h_{p_2} \otimes \cdots \otimes g_{p_d} \\
g_{p_1} \otimes g_{p_2} \otimes \cdots \otimes h_{p_d} & h_{p_1} \otimes f_{p_2} \otimes \cdots \otimes h_{p_d} &
\cdots & h_{p_1} \otimes g_{p_2} \otimes \cdots \otimes g_{p_d} \\
\cdots & \cdots & \cdots & \cdots \\
g_{p_1} \otimes h_{p_2} \otimes \cdots \otimes g_{p_d} & h_{p_1} \otimes g_{p_2} \otimes \cdots \otimes g_{p_d} &
\cdots & h_{p_1} \otimes h_{p_2} \otimes \cdots \otimes f_{p_d}
\end{bmatrix}. $$
\end{theorem}
\begin{theorem} \label{c_distro_nonnested}
The sequence of normalized matrices $\{n^{-2} A_{\nn,\pp}^{\spQQ_{\nn\bparam}} \}_n$ is distributed, in the sense of the eigenvalues, like the function
\begin{equation} \label{c_symbol_mD_nonnested}
\bnu (J^{-1} K({\bf G}) J^{-T} \circ H_{p_1,\ldots,p_d}^{\spQQ_\bparam}) \bnu^T,
\end{equation}
where
$$ \hspace{-0.5cm} H_{p_1,\ldots,p_d}^{\spQQ_\bparam} := \begin{bmatrix}
f_{p_1}^{\spQ_\param} \otimes h_{p_2}^{\spQ_\param} \otimes \cdots \otimes h_{p_d}^{\spQ_\param} &
g_{p_1}^{\spQ_\param} \otimes g_{p_2}^{\spQ_\param} \otimes \cdots \otimes h_{p_d}^{\spQ_\param} & \cdots & 
g_{p_1}^{\spQ_\param} \otimes h_{p_2}^{\spQ_\param} \otimes \cdots \otimes g_{p_d}^{\spQ_\param} \\
g_{p_1}^{\spQ_\param} \otimes g_{p_2}^{\spQ_\param} \otimes \cdots \otimes h_{p_d}^{\spQ_\param} & 
h_{p_1}^{\spQ_\param} \otimes f_{p_2}^{\spQ_\param} \otimes \cdots \otimes h_{p_d}^{\spQ_\param} & \cdots &
h_{p_1}^{\spQ_\param} \otimes g_{p_2}^{\spQ_\param} \otimes \cdots \otimes g_{p_d}^{\spQ_\param} \\
\cdots & \cdots & \cdots & \cdots \\
g_{p_1}^{\spQ_\param} \otimes h_{p_2}^{\spQ_\param} \otimes \cdots \otimes g_{p_d}^{\spQ_\param} & 
h_{p_1}^{\spQ_\param} \otimes g_{p_2}^{\spQ_\param} \otimes \cdots \otimes g_{p_d}^{\spQ_\param} & \cdots &
h_{p_1}^{\spQ_\param} \otimes h_{p_2}^{\spQ_\param} \otimes \cdots \otimes f_{p_d}^{\spQ_\param}
\end{bmatrix}. $$
\end{theorem}

\section{Conclusions} \label{sec:colconcl}

We constructed the matrices which follow from the use of an isogeometric collocation method based on uniform tensor-product GB-splines to approximate an elliptic PDE with variable coefficients and a generic geometry map, in a fashion which is similar to the one presented for isogeometric Galerkin methods in \cite{galerkin_gbsp}, but not limiting only to the case of constant coefficients (with the Laplacian operator for the principal terms) and a trivial geometry map.

Their spectral properties have been analyzed, by focusing on the spectral symbol $f$, which in accordance with previous results in Finite Differences and Finite Elements context, as well in IgA collocation with polynomial B-splines \cite{collocbsp04,Serra03,Serra06,our-MATHCOMP}, resulted to have a canonical structure incorporating the approximation technique, the geometry $ {\bf G} $, and the coefficients of the principal terms of the PDE $K$. 


The fact that (in the 1D case for simplicity) there is a monotonic and exponential convergence to zero of $f$ with respect to the degree $p$ at the points $ \theta = \pm \pi $, which is responsible for the slowdown, with respect to $p$, of standard iterative methods, have been addressed for the polynomial case in \cite{DGMSS14c}.

In the generalized case, some points remain still open. For example, the expression \eqref{c_symbol_mD_nested}, which appears in the study with polynomial B-splines and trigonometric or hyperbolic GB-splines in the nested case, is known to be positive, which is pivotal in order to devise truly robust numerical methods, but little is still known concerning \eqref{c_symbol_mD_nonnested}, which appears while trigonometric or hyperbolic GB-splines are used with the non-nested framework. A better understanding of these features would be very important in order to improve the possibility of numerical solving the resulting linear systems in a safe way.

\end{document}